\documentclass{amsart} 
\usepackage{amssymb}
\usepackage{amsfonts}
\usepackage{calrsfs}

 \usepackage{fullpage}
\usepackage{comment}

\font\te=eufm10
\newfont{\vlte}{eufm10 at 22pt}
\newfont{\lte}{eufm10 at 18pt}
\newfont{\smb}{msbm6}
\newfont{\mmb}{msbm8}
\newfont{\tmb}{msbm10}
\newfont{\lmb}{msbm10 at 18pt}

\newcommand{\ord}{\mbox{ord}}
\newcommand{\Gal}{\mbox{Gal}}

\newcommand{\mte}[1]{\mbox{\te {#1}}}

\newcommand{\be}{\begin{enumerate}}
\newcommand{\ee}{\end{enumerate}}
\newcommand{\su}{\subsection*}

\newcommand{\id}{\mbox{id}}

\newcommand{\rank}{\mbox{rank}}

\chardef\secsym=129
\newcommand{\calA}{{\mathcal A}}
\newcommand{\calB}{{\mathcal B}}
\newcommand{\calC}{{\mathcal C}}
\newcommand{\calD}{{\mathcal D}}
\newcommand{\calE}{{\mathcal E}}

\newcommand{\calL}{{\mathcal L}}
\newcommand{\calM}{{\mathcal M}}
\newcommand{\calN}{{\mathcal N}}

\newcommand{\calP}{{\mathcal P}}

\newcommand{\calR}{{\mathcal R}}
\newcommand{\calS}{{\mathcal S}}
\newcommand{\calT}{{\mathcal T}}
\newcommand{\calU}{{\mathcal U}}
\newcommand{\calV}{{\mathcal V}}
\newcommand{\calW}{{\mathcal W}}

\newcommand{\calZ}{{\mathcal Z}}

\newcommand{\C}{{\mathbb C}}
\newcommand{\F}{{\mathbb F}}

\newcommand{\Q}{{\mathbb Q}}
\newcommand{\R}{{\mathbb R}}
\newcommand{\Z}{{\mathbb Z}}

\newcommand{\pp}{{\mathfrak p}}

\newcommand{\qq}{{\mathfrak q}}

\newcommand{\ttt}{{\mathfrak t}}
\newcommand{\dd}{{\mathfrak d}}
\newcommand{\nn}{{\mathfrak n}}

\chardef\sha=88
\newtheorem{theorem}{Theorem}[section]
\newtheorem{lemma}[theorem]{Lemma}
\newtheorem{corollary}[theorem]{Corollary}
\newtheorem{proposition}[theorem]{Proposition}

\theoremstyle{definition}
\newtheorem{definition}[theorem]{Definition}

\theoremstyle{remark}
\newtheorem{remark}[theorem]{Remark}

\newtheorem{notation}[theorem]{Notation}
\newtheorem{notationassumption}[theorem]{Notation and Assumptions}

\newcommand{\norm}{{\mathbf N}}

\newcommand{\OO}{{\mathcal O}}

\begin{document}

\bibliographystyle{plain}%
  \title{Rings of Algebraic Numbers in Infinite Extensions of $\Q$ and Elliptic Curves Retaining Their Rank.}

\author{Alexandra Shlapentokh}
\thanks{The research for this paper has been partially supported by NSF grant DMS-0354907 and ECU
Faculty Senate Summer Research Grant.}
\address{Department of Mathematics \\ East Carolina University \\ Greenville, NC 27858}
\email{shlapentokha@ecu.edu }
\urladdr{www.personal.ecu.edu/shlapentokha}

\begin{abstract}%
We show that elliptic curves whose Mordell-Weil groups are finitely generated over some infinite extensions of $\Q$,
can be used to show the Diophantine undecidability of the rings of integers and bigger rings contained in some infinite
extensions of rational numbers.
\end{abstract}%
\maketitle%

\section{Introduction}
The interest in the questions of existential definability and decidability over rings goes back to a question that
was posed by Hilbert: given an arbitrary polynomial equation in several variables over $\Z$, is there a uniform
algorithm to determine whether such an equation has solutions in $\Z$? This question, otherwise known as Hilbert's
10th problem, has been answered negatively in the work of M. Davis, H. Putnam, J. Robinson and Yu. Matijasevich.
(See \cite{Da1}, \cite{Da2} and \cite{Mate}.) Since the time when this result was obtained, similar questions have
been raised for other fields and rings. In other words, let $R$ be a recursive ring. Then, given an arbitrary
polynomial equation in several variables over $R$, is there a uniform algorithm to determine whether such an
equation has solutions in $R$? One way to resolve the question of Diophantine decidability negatively over a ring
of characteristic 0 is to construct a Diophantine definition of $\Z$ over such a ring. This notion is defined
below.

\begin{definition}
Let $R$ be a ring and let $A \subset R$.  Then we say that $A$ has a Diophantine definition over
$R$ if there exists a polynomial $f(t,x_1,\ldots,x_n) \in R[t,x_1,\ldots,x_n]$ such that for any $t
\in R$,
\[
\exists x_1,\ldots,x_n \in R, f(t,x_1,...,x_n) = 0 \Longleftrightarrow t \in A.
\]

If the quotient field of $R$ is not algebraically closed,   we can allow
a Diophantine definition to consist of several polynomials without changing the nature of the
relation. (See \cite{Da2} for more details.)
\end{definition}

The usefulness  of Diophantine definitions stems from the following easy lemma.
\begin{lemma}
Let $R_1 \subset R_2$ be two recursive rings such that the quotient field of $R_2$ is not
algebraically closed. Assume that  Hilbert's Tenth Problem (abbreviated as ``HTP''in the future)   is
undecidable over $R_1$, and $R_1$ has a Diophantine definition over $R_2$. Then HTP is
undecidable over $R_2$.%
\end{lemma}

Using norm equations, Diophantine definitions have been obtained for $\Z$ over the rings of algebraic integers of
some number fields. Jan Denef has constructed a Diophantine definition of $\Z$ for the finite degree totally real
extensions of $\Q$. Jan Denef and Leonard Lipshitz extended Denef's results to all the extensions of degree 2 of
the finite degree totally real fields. Thanases Pheidas and the author of this paper have independently
constructed Diophantine definitions of $\Z$ for number fields with exactly one pair of non-real conjugate
embeddings. Finally Harold N. Shapiro and the author of this paper showed that the subfields of all the fields
mentioned above ``inherited'' the Diophantine definitions of $\Z$. (These subfields include all the abelian
extensions.) The proofs of the results listed above can be found in \cite{Den1},
\cite{Den3}, \cite{Den2}, \cite{Ph1}, \cite{Sha-Sh}, and \cite{Sh2}.\\ %

The author modified the norm method to obtain Diophantine definitions of $\Z$ for ``large'' subrings of totally real
number fields (not equal to $\Q$) and their extensions of degree 2.  (See \cite{Sh36},  \cite{Sh1}, \cite{Sh6},
and \cite{Sh3}.)  Further, again using norm equations, the author also showed that  in some totally real infinite
algebraic extensions of $\Q$ and extensions of degree 2 of such fields one can give  a Diophantine definition of
$\Z$ over the integral closures of ``small'' and ``large'' rings, though not over the rings of algebraic integers.
 (The terms ``large'' and ``small'' rings will be explained below.)   Unfortunately, the norm method, at least in
its present form, suffers from two serious limitations when used over infinite extensions: an ``infinite part'' of
the extension has to have a non-splitting prime (effectively requiring working over an infinite cyclic extension),
and one cannot use the method over the rings of algebraic integers.   At the same time, though,  one can describe
rather easily at least one big class of fields to which the method applies, e. g. all Abelian extensions with
finitely many ramified rational primes. \\%

Another method of constructing Diophantine definitions uses elliptic curves. The idea of using elliptic curves for this
purpose is due to Denef in \cite{Den2} where he showed that the following proposition held.
\begin{proposition}%
Let $K_{\infty}$ be a totally real algebraic possibly infinite extension of $\Q$.  If there exists an elliptic
curve ${\tt E}$ over $\Q$ such that $[{\tt E}(K):{\tt E}(\Q)] < \infty$, then $\Z$ has a Diophantine
definition over $O_K$.
\end{proposition}%
Expanding Denef's ideas, Bjorn Poonen proved the following result in \cite{Po}.
\begin{theorem}%
\label{thm:po}%
Let $M/K$ be a number field extension with an elliptic curve ${\tt E}$ defined over $K$, of rank one
over $K$, such that the rank of ${\tt E}$ over $M$ is also one.  Then $O_K$ (the ring of integers of $K$) is Diophantine over $O_M$.\\%
\end{theorem}%
 Cornelissen, Pheidas and Zahidi weakened somewhat assumptions of Poonen's theorem. Instead of requiring a rank 1
curve retaining its rank in the extension, they require existence of a rank 1 elliptic curve over the bigger field
and an abelian variety over the smaller field retaining its rank in the extension (see \cite{CPZ}).  Further,
Poonen and the author have independently shown that the conditions of Theorem \ref{thm:po} can be
weakened to remove the assumption that rank is one and require only that the rank in the extension
is positive and is the same as the  rank below (see \cite{Sh33} and \cite{Po3}).  In \cite{Sh33}
the author also showed that the elliptic curve technique can be used over large rings.

Elliptic curves of rank one have also been used to construct Diophantine models of $\Z$ (an alternative method for
showing Diophantine undecidability of a ring) over ``very large'' rings of rational and algebraic numbers, as well
as to construct infinite discrete in a $\pp$-adic and/or an archimedean topology Diophantine sets
over these rings. (See \cite{Po2} and \cite{PS}.) The interest in such sets has been motivated by a
series of conjectures by Barry Mazur  (see \cite{M1}, \cite{M2}, \cite{M3} and \cite{M4}) and their
variations (see \cite{Sh21}). \\

In this paper we explore to what extent the elliptic curve methods can be adapted for showing the Diophantine
undecidability of rings of algebraic numbers (including rings of integers) in infinite extensions.
If one uses elliptic curves instead of norm equations to construct Diophantine definitions over an
infinite extension, in principle, one does not need cyclic extensions and it is possible to work
over rings of integers.   The difficulties will lie along a different plane: the elliptic curve
method requires existence of an elliptic curve of positive rank with a finitely generated
Mordell-Weil group over an infinite extension.  While we already have plenty of examples of this
sort, the general situation is far from clear. (See \cite{M5} for more details.)

Another technical difficulty which occurs over infinite extensions is connected to defining bounds on the height of
the elements.  Using quadratic forms and divisibility we can solve the problem to large extent over totally real
fields and to some extent over extensions of degree 2 of totally real fields.

In this paper we refine our results on bounds as used in \cite{Sh17}, \cite{Sh26}, and \cite{Sh36}. We also generalize
Denef's results to any totally real infinite extension $K_{\infty}$ of $\Q$ and any of its extensions of degree 2
assuming some finite extension of $K_{\infty}$ has an elliptic curve of positive rank with a finitely generated
Mordel-Weil group. We will be able to treat rings of integers as well as ``large'' rings in these extensions.

We will also show that if there exists an elliptic curve of rank 1 with a  finitely generated Mordel-Weil group in an
infinite extension, then techniques from \cite{Po2} and \cite{PS} are adaptable for this situation to reach similar
results. Please note that for this application we will not need bound equations and thus the discussion can take place
over an arbitrary infinite algebraic extension of $\Q$.

\section{Preliminary Results, Definitions and the Statement of the Main Theorem.}%
In this section we state some technical propositions which will be used in the proofs and describe notation and
assumptions to be used in the sections below. We start with the proposition on  definability of integrality at finitely
many primes over number fields.

\begin{proposition}%
\label{prop:finmany}%
Let $K$ be a number field. Let ${\calW}_K$ be any set of primes of $K$. Let ${\calS}_K \subseteq
{\calW}_K$ be a finite set.  Let ${\calV}_K = {\calW}_K \setminus {\calS}_K$.  Then
$O_{K,\calV_K}$ has a Diophantine definition over $O_{K,\calW_K}$.  (See, for example, \cite{Sh5}.)\\%
\end{proposition}%
``Infinite'' versions of this proposition are more complicated. Before stating some of them below  we
introduce new terminology.%
\begin{notationassumption}%
\label{not:infversion}%
The following terminology will be used in the rest of the paper.
\begin{itemize}%
\item Let $L$ be an algebraic, possibly infinite extension of $\Q$. Let $Z$ be a number field contained in $L$. Let
$\calC_Z$ be a finite set of primes of $Z$. Assume further that there exists a polynomial $I_{\calC_Z/L}(x,
t_1,\ldots,t_k) \in Z[x,t_1,\ldots,t_k]$  such that
\begin{equation}
\label{eq:I}
I_{\calC_Z/L}(x, t_1,\ldots,t_k) =0
\end{equation}%
has solutions in $L$ only if $u_{\calC_Z/L}x$ is integral at all the primes of $\calC_Z$, where $u_{\calC_Z/L} \in
\Z_{>0}$ is fixed and depends only on $\calC_Z$.   Assume also that if $x \in Z$ and is integral at all the
primes of $\calC_Z$, then (\ref{eq:I}) has solutions in $Z$. Then  we will call $\calC_Z$-primes \emph{$L$-boundable}.
If we can set $u_{\calC_Z/L}=1$, then we will say that \emph{integrality is definable} at primes of $\calC_Z$ over $L$.
\item   Let $L$, as above,  be an algebraic, possibly infinite extension of $\Q$.  Let $q$ be a rational prime. Then let
the \emph{degree index of $q$} (with respect to $L$)  denoted by $i_L(q)$ be defined as follows:
\[%
i_L(q) = \max\{n \in \Z_{\geq 0} : n=\ord_{q}[M:\Q], \mbox{ where } M \mbox { is a number field
contained in } L\}
\]%
\end{itemize}
\end{notationassumption}%
The following statement is taken from Section 3 of \cite{Sh36}.%
\begin{proposition}%
\label{prop:inffinmany}
 Let $L$ be an algebraic, possibly infinite extension of $\Q$. Let $Z$ be a number field
contained in $L$ such that $L$ is normal over $Z$. Let $\calC_Z$ be a finite set of primes of $Z$ such that for
every $\pp_Z \in \calC_Z$ the following conditions are satisfied.
\begin{itemize}%
\item  There exists a non-negative integer $m_f$ such that any prime  lying  above $\pp_Z$ in a number field
contained in $L$ has a relative degree $f$ over $Z$ with $\ord_qf \leq m_f$.
\item There exists a non-negative integer $m_e$ such that any prime  lying above $\pp_Z$ in a number field
contained in $L$ has a ramification degree $e$ over $Z$ with $\ord_qe \leq m_e$.%
\end{itemize}%
Then   $\calC_Z$-primes are \emph{$L$-boundable}. If we also assume that the ramification degree for all the
factors of primes in $\calC_Z$ is bounded from above, then integrality at all the primes of $\calC_Z$ is definable.
\end{proposition}%
Finally we want to separate out a case which occurs quite often in our discussion of infinite extensions.
\begin{corollary}%
\label{cor:deg}%
Suppose $L$ is a normal, algebraic, possibly infinite extension of some number field such that for some odd
rational prime $q$ the degree index of $q$ with respect to $L$ is finite. Then for any number field $M \subset L$,
any $M$-prime $\pp_M$ is $L$-boundable.
\end{corollary}%

\begin{notationassumption}
Next we introduce several additional notational conventions to be used throughout the paper.
\begin{itemize}
\item Let $N$ be any finite extension of a  number field $U$. Let $\calT_U$ (or $\calV_U$, $\calW_U$, $\calS_U$,
$\calE_U$, $\calN_U$, $\calL_U$, $\calR_U$, \ldots) be any set of primes of $U$. Then let $\calT_N$  (or
$\calV_N$, $\calW_N$, $\calS_N$, $\calE_N$, $\calN_N$, $\calL_N$, $\calR_N$, \ldots) be the set of all primes of
$N$ lying above the primes of $\calT_U$.%
\item If $N_{\infty}$ is an algebraic, possibly infinite extension of $U$, then $O_{N_{\infty},
\calT_{N_{\infty}}}$ ( or $O_{N_{\infty}, \calV_{N_{\infty}}}$, $O_{N_{\infty}, \calW_{N_{\infty}}}, \ldots$) will
denote the integral closure of $O_{U, \calT_U}$ in $N_{\infty}$ (or respectively of  $O_{U, \calV_U}$, $O_{U,
\calW_U}, \ldots$).
\item For any number field $U$ let $\calP(U)$ denote the set of all non-archimedean primes of $U$.%
\item For any field $U$ and a set $\calW_U \subset \calP(U)$, let $\overline \calW_U$ be the
closure of $\calW_U$
with respect to conjugation over $\Q$.  %
\item For any field $U$ and a set $\calW_U \subset \calP(U)$, let $\hat{\calW}_U$ be a subset of $\calW_U$
obtained from $\calW_U$ by removing a prime of highest relative degree over $\Q$ from every complete set of
$\Q$-conjugates contained in $\calW_U$.
\item We will assume that all the fields under discussion are subfields of $\C$.  Given two fields $U, T \subset
\C$ we will interpret $UT$ to mean the smallest subfield of $\C$ containing both fields.
\item For a number field $K$ we will denote its Galois closure over $\Q$ by $K^{\Gal}$.
\end{itemize}
\end{notationassumption}%
Below we state two well-known technical propositions which are also quite important for the proofs in this paper.%
\begin{proposition}%
Let $K$ be a number field. Let ${\mathcal W}_K$ be any set of primes of $K$.  Then the set of non-zero elements of
$O_{K,\calW_K}$ has a Diophantine definition over $O_{K,\calW_K}$.   (See, for example, \cite{Sh5}.)\\%
\end{proposition}

This proposition allows us to use variables which take values in $K$ while we are ``officially'' working with
variables taking values in $O_{K,{\mathcal W}_K}$.  We write these $K$-variables as ratios of variables in
$O_{K,{\mathcal W}_K}$ with the proviso that the denominator is not zero.\\%

The next proposition allows us to establish some bounds on real valuations.  It is due to Denef and can be found in
\cite{Den3} or Section 5.1 of \cite{Sh34}.
\begin{proposition}
\label{prop:realembed}%
Let $K_{\infty}$ be a totally real algebraic possibly infinite extension of $\Q$. Let $G_{\infty}$ be a finite
extension of $K_{\infty}$ generated by an element $\alpha \in K_{\infty}$. Then the set $\{x \in G_{\infty}:
\sigma(x) \geq 0\, \, \forall \sigma : G_{\infty} \longrightarrow \R\}$ is Diophantine over
$G_{\infty}$.
\end{proposition}%

The following proposition allows us to avoid certain sets of primes in the numerators and can be derived from
Proposition 25, Section 8, Chapter I of \cite{L}.
\begin{lemma}%
\label{le:negorder}%
Let $T/K$ be a number field extension. Let $R(X)$ be the monic irreducible polynomial of an integral generator of
$T$ over $K$. Let $\calV_K$ be the set of all primes of $K$ without relative degree one factors in $T$ and not
dividing the discriminant of $R(X)$. Then for all $x \in K$ and all $\pp \in \calV_K$ we have that $\ord_{\pp}R(x)
\leq 0$.
\end{lemma}%

The next lemma will allow us to conclude that under some circumstances the set of primes we can prevent from
appearing in the numerators of the divisors of the elements is closed under conjugation over $\Q$.
\begin{lemma}%
\label{le:conjugates}%
Let $T, K, R(X)$ be as in Lemma \ref{le:negorder}. Let $T_0$ be a finite extension of $\Q$ and assume that
$T=T_0K$. Assume further that $[T_0:\Q]=[T:K] \geq 2$,  $T/\Q$ is Galois, and $[K:\Q]$ is Galois. Suppose that for
some prime $\pp_K$ of $K$ we have that for all $x \in K$ it is the case that  $\ord_{\pp_K} R(x) \leq 0$.  Then for
any $\bar{\pp}_K$- conjugate of $\pp_K$ over $\Q$, for all $x \in K$, we have that $\ord_{\bar{\pp}_K}R(x) \leq 0$
\end{lemma}%
\begin{proof}%
The lemma follows almost immediately from the fact that $R(X)$, the monic irreducible polynomial of an integral
generator of $T$ over $K$, can be assumed to have rational integer coefficients.
\end{proof}%

We generalize this result for some classes of infinite extensions.%
\begin{lemma}%
\label{le:boundable}%
Let $T, T_0, K, R(X), \calV_K$ be as in Lemmas \ref{le:negorder} and \ref{le:conjugates}, assume
that $[T_0:\Q] >2$ and $T/K$ is Galois. Let $K_{\infty}$ be a normal possibly infinite extension of
$K$ such that for any number field $N \subset K_{\infty}$ with $K \subseteq N$, we have that
$([N:K], [T:K])=1$. (Observe that this condition implies by Proposition \ref{prop:inffinmany} that
any finite set of primes of  $K$ is boundable over $K_{\infty}$.) Let $\calW_K$ be the set of all
primes of $K$ without degree one factors in $T$. Let $\calS_K$ be any finite set of primes of $K$.
Let $\calE_K=(\calW_K\cup \calS_K) \setminus \calV_K$. Let $\mu$ - a generator of $T_0$ over $\Q$
and $T$ over $K$ be an integral unit. Let $a \in \Z_{>0}$ be an integer divisible by all the primes
in $\calE_K$. Let $M \subset K_{\infty}$ be any number field containing $K$ and the values of all
the variables below. Finally assume that the following equations hold over $K_{\infty}$.
\begin{equation}
\label{eq:int1}
\left \{
\begin{array}{c}
I_{\calE_K/K_{\infty}}(x, t_1,\ldots,t_k)=0\\
y=R(u_{\calE_K/K_{\infty}}x)
\end{array}
\right .
\end{equation}
Then for any $\pp_M \in \calW_M \cup \calS_M$ we have that $\ord_{\pp_M}y \leq 0$.
\end{lemma}%
\begin{proof}%
Without loss of generality we can assume that $M/K$ is a Galois extension.  (If not we can take the Galois closure
of $M$ over $K$ contained in $K_{\infty}$.)  Next by Lemma \ref{le:nodegonefact} we have that no prime $\pp_M$ of
$M$ has a relative degree one factor in $MT$.  Thus for all primes $\pp_M \in \calW_M \setminus \calE_M$ we have
that $\ord_{\pp_M}y \leq 0$.  Next we proceed to the primes of $\calE_M$.  By definition of
$I_{\calE_K/K_{\infty}}(x, t_1,\ldots,t_k)$ and $u_{\calE_K/K_{\infty}}$ we have that for any $\qq_M \in \calE_M$
it is the case that $\ord_{\qq_M}u_{\calE_K/K_{\infty}}x \geq 0$.  Thus by choice of $a$ we also have that
$\ord_{\qq_M}au_{\calE_K/K_{\infty}}x >0$.  Since $\mu$ is an integral unit, the free term of $R(X)$ is $\pm1$ and
thus $R(au_{\calE_K/K_{\infty}}x) \equiv \pm1 \mod \qq_M$.
\end{proof}%

To state the main theorem of the paper we need the following definitions.

\begin{definition}%
Let $K$ be a number field and let $\calW_K \subset \calP(K)$. Let $O_{K,\calW_K}=\{x \in K:
\ord_{\pp}x\geq 0\,\, \forall \pp \not \in \calW_K\}$. Then call $O_{K,\calW_K}$ a \emph{small}
ring if $\calW_K$ is finite. If $\calW_K$ is infinite, then call the ring \emph{big} or
\emph{large}. If $K_{\infty}$ is an infinite algebraic extension of $K$, then call the integral
closure of a small subring of $K$ in $K_{\infty}$ small, and call the integral closure of a big
subring of $K$ in $K_{\infty}$ big.
\end{definition}%

\begin{remark}%
Note that if $\calW_K=\emptyset$, then $O_{K,\calW_K}=O_K$ is the ring of integers of $K$, and if
$\calW_K=\calP(K)$, then  $O_{K,\calW_K}=K$.  The small rings are also known as ``rings of
$\calS$-integers''.  Observe  that the integral closure of a small ring in a finite extension is
also small, and similarly, the integral closure of a big ring in a finite extension is also big.
\end{remark}%

\begin{definition}%
Let $K_{\infty}$ be an algebraic possibly infinite extension of $\Q$.  Let $R$ be a small subring of $K_{\infty}$.
Suppose now that there exists a number field $K$ contained in $K_{\infty}$ and finite set of primes $\calS_K$ of
$K$ such that $R=O_{K_{\infty},\calS_{K_{\infty}}}$ \emph{and} all the primes of $\calS_K$ are either
boundable in $K_{\infty}$ or integrality at all the primes of $\calS_K$ is definable over $K_{\infty}$, then we
will say that the set of primes occurring in the denominators of the divisors of elements of $R$ is boundable or
that integrality is definable at the primes occurring in the denominators of divisors of the elements of $R$.
\end{definition}%

We are now ready to state the main theorems of our paper.
\su{Main Theorem A}%
Let $K_{\infty}$ be a totally real possibly infinite algebraic extension of $\Q$. Let $U_{\infty}$
be a finite extension of $K_{\infty}$ such that there exists an elliptic curve ${\tt E}$ defined
over $U_{\infty}$ with ${\tt E}(U_{\infty})$ finitely generated and of a positive rank. Then $\Z$ is
existentially definable and HTP is unsolvable over the ring of integers of $K_{\infty}$. (See
Theorem \ref{thm:mainint}.)%

\su{Main Theorem B}%
Let $K_{\infty}, U_{\infty}$ and ${\tt E}$ be as above. Let $G_{\infty}$ be an extension of degree
two of $K_{\infty}$. If $G_{\infty}$ has no real embeddings into it algebraic closure, assume
additionally that $K_{\infty}$ has a totally real extension of degree two. Then $\Z$ is
existentially definable and HTP is unsolvable over the ring of integers of $G_{\infty}$. (See
Theorem \ref{thm:mainintdeg2}.) %

\su{Main Theorem C}%
Let $K_{\infty}$ be a totally real possibly infinite algebraic extension of $\Q$, normal over some number field and
with a finite  degree index  for some odd rational prime number $p$.  Let $U_{\infty}$ be a finite
extension of $K_{\infty}$ such that there exists an elliptic curve ${\tt E}$ defined over $U_{\infty}$ with
${\tt E}(U_{\infty})$ finitely generated and of a positive rank. Then%
\be%
\item $K_{\infty}$ contains a large subring $R$ such that $\Z$ is definable over $R$ and HTP is unsolvable over $R$.%
\item For any small subring $R$ of $K_{\infty}$ we have that $\Z$ is definable over $R$ and HTP is unsolvable over
$R$.
\ee%
(See Theorems \ref{thm:mainnonint0} and \ref{thm:smallring}.)

\su{Main Theorem D}%
 Let $K_{\infty}$ be a totally real possibly infinite algebraic extension of $\Q$ normal over
some number field and such that there exist infinitely many  rational prime numbers of finite degree index with
respect to $K_{\infty}$.   Let $U_{\infty}$ be a finite extension of $K_{\infty}$ such that there exists an
elliptic curve ${\tt E}$ defined over $U_{\infty}$ with ${\tt E}(U_{\infty})$ finitely generated
and of a positive rank.  Then for any $\varepsilon  >0$ we have that $K_{\infty}$ contains a large
subring $R$ satisfying the following conditions: %
\be%
\item There exists a number field $K \subset K_{\infty}$ and a set $\calW_K \subset \calP(K)$ of (Dirichlet or
natural) density greater than $1-\varepsilon$ such that $R=O_{K_{\infty},\calW_{K_{\infty}}}$.%
\item $\Z$ is definable over $R$ and HTP is unsolvable over $R$.%
\ee%
   (See Theorem \ref{thm:mainnonint}.)
\begin{remark}%
Note that in this paper we will have no assumptions on the nature of the totally real field besides the assumption
on the existence of the elliptic curve satisfying the conditions above.  Thus we will be able to
consider a larger class of fields than in  \cite{Sh17}, \cite{Sh26},  and \cite{Sh36}.
\end{remark}%
\su{Main Theorem E}

Let $K_{\infty}$ be a totally real possibly infinite algebraic extension of $\Q$ normal  over some
number field $K$,  with an odd rational prime $p > [K^{\Gal}:\Q]$ of 0 degree index relative to
$K_{\infty}$, and with a 0 degree index for 2. Let $U_{\infty}$ be a finite extension of
$K_{\infty}$ such that there exists an elliptic curve ${\tt E}$ defined over $U_{\infty}$ with
${\tt E}(U_{\infty})$ finitely generated and of a positive rank. Let $GK_{\infty}$ be an extension
of degree 2 over $K_{\infty}$.
\be%
\item $GK_{\infty}$ contains a big ring where $\Z$ is existentially definable and HTP is unsolvable. %
\item $\Z$ is definable and HTP is unsolvable over any small subring of $GK_{\infty}$.%
\item If we assume additionally that the set of rational primes with 0 degree index with respect to
$K_{\infty}$ is infinite, then for every $\varepsilon >0$ there exist a number field $K \subset K_{\infty}$ and a
set $\calZ_K \subset \calP(K)$ of density (natural or Dirichlet) bigger than $1/2 -\varepsilon$ such that $\Z$ is
existentially definable and HTP is unsolvable in the integral closure of $O_{K,\calZ_K}$ in $G_{\infty}$.
\ee%
(See Theorems \ref{thm:bigrings2} and \ref{thm:bigrings3}.)
\su{Main Theorem F}%
Let $K_{\infty}$ be an algebraic extension of $\Q$ such that there exists an elliptic curve ${\tt E}$ defined over
$K_{\infty}$ with ${\tt E}(K_{\infty})$ of rank 1 and finitely generated.  Fix a Weierstrass
equation for ${\tt E}$ and a number field $K$ containing all the coefficients of  the Weierstrass
equation and the coordinates of all the generators of ${\tt E}(K_{\infty})$.  Assume that $K$ has
two odd relative degree one primes $\pp$ and $\qq$ such that integrality is definable at $\pp$ and
$\qq$ over $K_{\infty}$.
\be%
\item There exist a set $\calW_K$ of $K$-primes of natural density 1 such that over
$O_{K_{\infty},\calW_{K_{\infty}}}$ there exists an infinite Diophantine set
simultaneously discrete in all archimidean and non-archimedean topologies of $K_{\infty}$.

\item There exist a set $\calW_K$ of $K$-primes of natural density 1 such that over
$O_{K_{\infty},\calW_{K_{\infty}}}$ there exists a Diophantine model of $\Z$ and therefore HTP is
not solvable over $O_{K_{\infty},\calW_{K_{\infty}}}$.
\ee%

(See Theorem \ref{thm:rank1}.)

\section{Bounds and Their Uses.}%
\setcounter{equation}{0}%
We start with another notation set and some terminology.%
\begin{notationassumption}%
\label{not:0}%
\begin{itemize}%
We will use the following notation throughout the rest of the paper. %

\item Let $T$ be any field and let $t \in T$. Then let $\displaystyle \mte{d}_T(t) = \prod_{\mte{p}
\in \calP(T)}\mte{p}^{a(\mte{p})}$, where the product is taken over all primes $\mte{p}$ of $T$
such that $-a(\mte{p})=\ord_{\mte{p}}t <0$. Further, let $\mte{n}_T(t) = \mte{d}_N(t^{-1})$.%
\item Let $G/U$ be any number field extension. Let $\displaystyle \mte{A}=\prod_{\pp_i \in \calP(G)}
\pp_i^{n_i}$ be an integral divisor of $G$ such that $\mte{A}$ is equal to the factorization in $G$
of some integral divisor of $U$. Then we
will say that `` $\mte{A}$ can be considered as an integral divisor of $U$''. %
\item Let $T$ be a number field. Let $\mte{A}, \mte{B}$ be integral divisors of $T$ such that for any $\pp \in
\calP(T)$ we have that $\ord_{\pp}\mte{A} \leq \ord_{\pp}\mte{B}$. Then we will say that ``$\mte{A}$
divides $\mte{B}$ in the semigroup of integral divisors of $T$'' or simply ``$\mte{A}$ divides $\mte{B}$''.%
\item Let $T$ be a number field and let $\mte{A}, \mte{B}$ be  divisors of $T$ such that $\mte{A}=\mte{B}^2$.  Then
by $\sqrt{\mte{A}}$ we will mean $\mte{B}$.%

\end{itemize}%
\end{notationassumption}%

Next we prove two technical lemmas dealing with bounds on valuations.
\begin{lemma}%
\label{le:numerators} %
Let $T$ be a number field, let $\calW_T \subset \calP(T)$. Let $x \in O_{T, \overline \calW_T}, z
\in O_{T, \calW_T}, xz\not = 0$. Assume that $\mte{n}_T(x)$ divides $\mte{n}_T(z)$ in the semigroup
of integral divisors of $T$. Let $X, Y, Z, W \in \Z_{>0}$ be such that $(X,Y)=1, (Z,W)=1$,
$\displaystyle \left |{\mathbf N}_{T/\Q}(x)\right |=\frac{X}{Y}$, and $\displaystyle \left |{\mathbf
N}_{T/\Q}(z)\right |=\frac{Z}{W} $. Then $\displaystyle \frac{Z}{X} \in \Z_{>0}$.
\end{lemma}%
\begin{proof}%
Let $\mathfrak{Z}_1, \mte{Z}_2, \mathfrak {W}, \mathfrak{X}, \mathfrak {Y}$ be integral divisors of $T$ such that
${\mathfrak Z}_1$  and $\mathfrak{X}$ are composed of the primes outside $\overline \calW_T$, while
$\mte{Z}_2$, $\mathfrak Y, \mathfrak W$ are composed of primes in $\overline \calW_T$,
$\mathfrak{Z}_1$, $\mte{Z}_2$, and $\mathfrak {W}$ are pairwise relatively prime, $\mathfrak{X}$
and $\mathfrak {Y}$ are relatively prime, $\displaystyle \frac{\mathfrak{Z}_1\mte{Z}_2}{\mathfrak
W}$ is a divisor of $z$, and $\displaystyle \frac{\mathfrak X}{\mathfrak Y}$ is a divisor of $x$.
Since $\overline \calW_T$ is closed under conjugation over $\Q$, we conclude that ${\mathbf
N}_{T/\Q}({\mathfrak X})$ and ${\mathbf N}_{T/\Q}({\mathfrak Y})$ have no common factors as rational
integers. Similarly, there are no rational primes occurring simultaneously in ${\mathbf
N}_{T/\Q}({\mathfrak Z}_1)$ and ${\mathbf N}_{T/\Q}({\mathfrak W})$. So we can conclude that
$X={\mathbf N}_{T/\Q}({\mathfrak X})$ (as divisors of $\Q$), and $ {\mathbf N}_{T/\Q}({\mathfrak
Z}_1)$  divides $Z$ (as divisors of $\Q$). Further, by assumption,
$\displaystyle \frac{\mathfrak{Z}_1}{\mathfrak{X}}$ is an integral divisor. Thus,
$\displaystyle\frac{{\mathbf N}_{T/\Q}({\mathfrak Z}_1)}{{\mathbf N}_{T/\Q}({\mathfrak X})}$ is
also an integral divisor. In other words, $\displaystyle \frac{Z}{X}$ is an integer.
\end{proof}%
\begin{lemma}%
\label{le:denominators} %
Let $T$ be a number field. Let $\calW_T$ be a set of primes of $T$. Let $x_1 \in O_{T,\overline
\calW_T}, x_1\not = 0$ be such that $x_1$ does not have a positive order at any prime of $\overline
\calW_T$. Let $X, Y\in \Z_{>0}$ be such that $(X,Y)=1,   \left |{\mathbf N}_{T/\Q}(x_1)\right
|=\frac{X}{Y}$.  Let $x_2 \in O_{T,\overline \calW_T}$ be a conjugate of $x_1$ over $\Q$.  Then
$Y^2|{\mathbf N}_{T/\Q}(x_1-x_2)| \in \Z_{>0}$.
\end{lemma}%
\begin{proof}%
Let $\frac{\mte{X}_1}{\mte{Y}_1}, \frac{\mte{X}_2}{\mte{Y}_2}, \frac{\mte{U}}{\mte{V}}$  be the divisors of $x_1,
x_2$ and $x_1-x_2$ respectively.  Observe that on the one hand, ${\mathbf N}_{T/\Q}(x_1)={\mathbf
N}_{T/\Q}(x_2)=\frac{X}{Y}$ and, using the same argument as in the proof of Lemma \ref{le:numerators}, we have
that  ${\mathbf N}_{T/\Q}(\mte{Y}_1)={\mathbf N}_{T/\Q}(\mte{Y}_2)=Y$.  On the other hand, $\mte{V}$ divides
$\mte{Y}_1\mte{Y}_2$ in the semigroup of  integral divisors and therefore, ${\mathbf N}_{T/\Q}(\mte{V})$ divides
$Y^2$ in the semigroup of integral divisors of $\Q$.   Further, if we let $\displaystyle{\mathbf
N}_{T/\Q}(x_1-x_2)=\frac{A}{B}$, with $A$ and $B$ being relatively prime integers, then $B$ divides ${\mathbf
N}_{T/\Q}(\mte{V})$ in the integral divisor semigroup of $\Q$.  Thus, $B$ divides $Y^2$ in $\Z$.
\end{proof}%
Now consider the following use of bounds:
\begin{proposition}%
\label{prop:usebounds}%
Let $U$ be a number field. Let $\calW_U \subset \calP(U)$. Let $p$ be a rational prime such that no factor of $p$
is in $\calW_U$. Let $x, z$ be elements of the algebraic closure of $U$ and assume that for any prime $\qq$ of
$U(x,z)$ lying above a prime of $\overline \calW_U$ we have that $\ord_{\qq}x\leq 0$ and
$\ord_{\qq}z \leq 0$. Let $T$ be the Galois closure of $U(x,z)$. Let $t \in
O_{U,\calW_U}$. Assume also that $\mte{n}_T(z)$ can be considered as a divisor of $U$. Finally
assume that the following equations and conditions are satisfied over $O_{T,\calW_T}$, where
$\id=\sigma_1,\ldots, \sigma_{[T:\Q]}$ are all the distinct embeddings of $T$ into $\C$.
\begin{equation}%
\label{eq:1}
|\sigma_i(x)| \leq |\sigma_i(z)|, i=1,\ldots,[T:\Q],
\end{equation}
\begin{equation}
\label{eq:2}
|\sigma_i(x)| \geq 1, i=1,\ldots,[T:\Q],
\end{equation}
\begin{equation}
\label{eq:2.1}
|\sigma_i(z)| > 1, i=1,\ldots,[T:\Q],
\end{equation}
\begin{equation}%
\label{eq:3}
z \equiv 0 \mod x \mbox{ in } O_{T,\calW_T},
\end{equation}%
\begin{equation}
\label{eq:4}
x \equiv t \mod pz^4 \mbox{ in } O_{T,\calW_T}.
\end{equation}
Then $x \in  O_{U,\calW_U}$.
\end{proposition}%
\begin{proof}%
Using notation of Lemma \ref{le:numerators}, let $\displaystyle \left |{\mathbf N}_{T/\Q}(x)\right |
=\frac{X}{Y}$, where $(X,Y)$ are relatively prime positive integers, and let $\displaystyle\left
|{\mathbf N}_{T/\Q}z\right | = \frac{Z}{W}$, where $Z, W \in \Z_{>0}$ and $(Z,W)=1$ in $\Z$. Then
by Lemma \ref{le:numerators} we have that  $X$ divides $Z$ in $\Z$ and therefore $1 \leq X \leq Z$.
Further, from equation (\ref{eq:2}), we also have that
\begin{equation}%
\label{eq:4.01}
Y \leq X \leq Z.
\end{equation}%
Let $\mte{Z}= \mte{n}_T(z)$. From (\ref{eq:4}) we have that $p\mte{Z}^4$ divide $\nn_T(x-t)$. Now,
assume that $\hat{x}$ is a conjugate of $x$ over $U$. Then $p\mte{Z}^4$ divides $\nn_T(x- \hat{x})$
in the semigroup of integral divisors of $T$. Let
\begin{equation}%
\label{eq:4.1}
\left |{\mathbf N}_{T/\Q}(x-\hat{x})\right | = \frac{A}{B},
\end{equation}%
where $A, B$ are relatively prime positive integers.  Then on the one hand, by Lemma \ref{le:numerators} again, we
have that either $x=\hat{x}$ or $p^{[T:\Q]}Z^4$ divides $A$ and therefore%
\begin{equation}%
\label{eq:4.2}
p^{[T:\Q]}Z^4 \leq A.
\end{equation}%
On the other hand, from equation (\ref{eq:1}) we have that the absolute value of any conjugate of $x-\hat{x}$
over $\Q$ is less than 2 times the absolute value of the corresponding conjugate of $z$. Thus,
 \begin{equation}%
 \label{eq:5}%
 | {\mathbf N}_{T/\Q}(x-\hat{x})| \leq 2^{[T:\Q]}|{\mathbf N}_{T/\Q}(z)|.
 \end{equation}%
 Using equation (\ref{eq:4.1}) we can now write
 \begin{equation}%
 \label{eq:6}%
 \frac{A}{B} \leq 2^{[T:\Q]}|{\mathbf N}_{T/\Q}(z)|
 \end{equation}%
 and so
 \begin{equation}%
 \label{eq:7}
 A \leq 2^{[T:\Q]}B|{\mathbf N}_{T/\Q}(z)|.
 \end{equation}%
 Thus, combining (\ref{eq:7}), (\ref{eq:4.2}), (\ref{eq:4.01}) and using Lemma \ref{le:denominators} we get
 \begin{equation}%
 \label{eq:8}%
p^{[T:\Q]}Z^4 \leq 2^{[T:\Q]}B|{\mathbf N}_{T/\Q}(z)|=2^{[T:\Q]}B\frac{Z}{W} \leq 2^{[T:\Q]}Y^2\frac{Z}{W} \leq
2^{[T:\Q]}Z^3.
\end{equation}%
Since $p \geq 2$ and from equation (\ref{eq:2}) we know that $Z >1$, the last inequality cannot be true. Thus,
$x=\hat{x}$. Since $\hat{x}$ was an arbitrary conjugate of $x$ over $U$, we must conclude that $x \in U$.
\end{proof}%

The bounds also come in the following lemma which we will use below.  It is a slight modification of Lemma 5.1 of
\cite{Sh26}.
\begin{lemma}%
\label{le:modify}%
Let $T/U$ be an extension of number fields. Let $x \in T$ and let $\mte{T}=\mte{n}_{T}(x)$.  Let $\mte{A}$ be an
integral divisor of $U$ such that $\mte{T}^2$ divides $\mte{A}$ in the semigroup of the integral divisors of $T$.
Let $t \in U$ be such that  $\mte{A}$ divides $\mte{M}=\mte{n}_{T}(x-t)$ in the semigroup of integral divisors of
$T$.  Then $\mte{T}$ can be considered as a divisor of $U$.
\end{lemma}%
\begin{proof}%
We will show that for all $\pp \in \calP(T)$, such that $\ord_{\pp}x >0$ we have that $\ord_{\pp}x = \ord_{\pp}t$,
and for any $\qq$ conjugate to $\pp$ over $U$ we also have that $\ord_{\qq}x = \ord_{\qq}t >0$. Indeed, let $\pp$
be a prime of $T$ such that $\ord_{\pp}x >0$. Then given our assumptions on $\mte{A}$ we have that $\ord_{\pp}(x -
t) >\ord_{\pp}x$. The only way this can happen is for $\ord_{\pp}x=\ord_{\pp}t >0$. Next note that if $\qq$ is a
conjugate of $\pp$ over $U$, then $\ord_{\qq}\mte{A} >0$ implying that $\ord_{\qq}(x-t) >0$ and since $t \in T$,
we also have that $\ord_{\qq}t >0$. Thus, $\ord_{\qq}x >0$ and as above $\ord_{\qq}x = \ord_{\qq}t>0$. Thus,
$\mte{T}$ can be viewed as a divisor of $U$.
\end{proof}%


\section{Properties of Elliptic Curves}
\label{sec:ec} In this section we go over some properties of elliptic curves which will allow us to make sure we
can satisfy equivalencies of the form (\ref{eq:3}) and (\ref{eq:4}).
\begin{notation}%
\label{not:1}
We start with a notation set to be used below.
\begin{itemize}%
\item Let $U$ be a number field.%
\item Let ${\tt E}$ denote an elliptic curve defined over $U$ -- i.e. a non-singular cubic curve
whose affine part is given by a fixed Weierstrass equation with coefficients in $O_U$. (See
III.1 of \cite{Sil1}.)%
 \item For any field $T$ and any  $m \in \Z_{\geq 0}$ let ${\tt E}(T)[m]$ be
the group of $m$-torsion
points of ${\tt E}(T)$. By ``0-torsion'' we mean the identity of ${\tt E}$.%
 \item If $Q \in {\tt E}(U)$ is any point, then $(x(Q),y(Q))$ will denote the
affine coordinates of $Q$ given by the Weierstrass equation above.
\item Let $P$ be a fixed point of infinite order. %
\item Let $U_{\infty}/U$ be an infinite abelian Galois extension.
\item Assume $\rank ({\tt E}(U_{\infty}))=\rank ({\tt E}(U))$.%
\end{itemize}%
\end{notation}%

\begin{lemma}%
\label{le:Po1}%
If $I \subset O_U$ is a nonzero ideal. Then there exists a non-zero multiple $[l]P$ of $P$ such that $I \Big{|}
\mte{d}(x([l]P))$.%
\end{lemma}
\begin{proof}%
This lemma follows immediately from Lemma 10 of \cite{Po} even though we no longer assume that the curve is of
rank 1. The proof is unaffected by this change.
\end{proof}%
\begin{lemma}%
\label{le:Po2}%
There exists a positive integer $r$ such that for any positive integers $l, m$,
\[%
\mte{d}_U(x([lr]P) \Big{|} \mte{n}_U(\frac{x[lr](P)}{x([mlr]P)}-m^2)^2.%
\]%
\end{lemma}%
\begin{proof}%
Let $r$ be a positive integer defined in Lemma 8 of \cite{Po}.  Then the statement above follows immediately
from Lemma 11 of \cite{Po}.  The proof is again unaffected by the fact that we no longer assume ${\tt E}$ to be of
rank 1.%
\end{proof}%
\begin{lemma}%
\label{le:Po3}%
Let $r$ be as in Lemma \ref{le:Po2}.  Let $Q',Q \in [r]{\tt E}(U) \setminus\{O\}, Q' =[k]Q$.  Then
$\mte{d}_U(x(Q))$ divides $\mte{d}_U(x(Q'))$ in the semigroup of integral divisors of $U$.
\end{lemma}%
\begin{proof}%
See Lemma 9 of \cite{Po}.%
\end{proof}%

\begin{lemma}%
\label{le:ratio}%
Let $Q,Q'$ be as in Lemma \ref{le:Po3}. Then $\mte{d}_U(x(Q))$ and $\displaystyle \mte{d}_U\left
(\frac{x(Q)}{x(Q')}\right )$ do not have any common factors.
\end{lemma}%
\begin{proof}%
By Lemma \ref{le:Po3} we know that $\displaystyle \frac{\mte{d}_U (x(Q'))}{\mte{d}_U(x(Q))}=\mte{A}$
is an integral divisors. Next we note that the divisor of $\displaystyle \frac{x(Q)}{x(Q')}$ is of
the form
\[%
\frac{\mte{n}_U(x(Q))\mte{d}_U(x(Q'))}{\mte{d}_U(x(Q))\mte{n}_U(x(Q'))}=
\frac{\mte{A}\mte{n}_U(x(Q))}{\mte{n}_U(x(Q'))}
\]%
so that $\displaystyle\mte{d}_U\left (\frac{x(Q)}{x(Q')}\right )$ divides $\mte{n}_U(x(Q'))$ in the
group of integral divisors of $U$. Now $\mte{n}_U(x(Q'))$ has no common factors with
$\mte{d}_U(x(Q'))$ and thus with $\mte{d}_U(x(Q))$. Consequently,
$\displaystyle\mte{d}_U\left(\frac{x(Q)}{x(Q')}\right )$ has no common factors with
$\mte{d}_U(x(Q))$
\end{proof}%

The last lemma of this section follows from the chosen form of the Weierstrass equation.

\begin{lemma}%
\label{le:evenorder}%
$\mte{d}_U(x(Q))$ is a square of an integral divisor of $U$.
\end{lemma}%

The next two propositions will provide foundations for a construction of examples of elliptic curves with
finitely generated groups in some infinite extensions.  We first state a theorem which is a special case of
Theorem 12 from \cite{Sil3}.  For our special case below we can set the parameter $[F(nP):F]$ to 1.
\begin{theorem}
\label{thm:silverman}%
Let $F/\Q$ be a number field, let ${\tt E}/F$ be an elliptic curve that does not have complex multiplication. Then there
is an integer $k=k({\tt E}/F)$ so that for any point $P \in {\tt E}(\tilde{\Q})$, where $\tilde \Q$
is the algebraic closure of $\Q$, with $F(P)/F$ abelian we have that $[F(P):F]$ divides
$k[F(nP):F]$ for all $n$.
\end{theorem}%
We now use the theorem to prove that ${\tt E}(U_{\infty})$ is finitely generated.  The proof of the
 proposition was suggested to the author by Karl Rubin.
\begin{proposition}%
\label{prop:fingen}
${\tt E}(U_{\infty})$ is finitely generated.
\end{proposition}%
\begin{proof}%
Let $k$ be a positive integer defined in Theorem \ref{thm:silverman}.   Since $\rank({\tt E}(U_{\infty}))=\rank({\tt E}(U))$,
for any $Q\in {\tt E}(U_{\infty})$ for some $m \in \Z_{>0}$ we have that $[m]Q \in {\tt E}(U)$.  Further, since
$U_{\infty}/U$ is abelian, we also have that $U(Q)/U$ is abelian.  Thus, $m$ divides $k$ and
${\tt E}(U_{\infty})$ is finitely generated.
\end{proof}%

\section{Diophantine Definitions of  Rational Integers  for the Totally Real Case.}
\label{sec:def}%
\setcounter{equation}{0}%
In this section we construct  Diophantine definitions of   $\Z$ and  \emph{some} rings of rational $\calS$-integers
with finite $\calS$ which we also called the small rings.
\begin{notationassumption}%
\label{not:integers}
We add the following notation and assumptions to our notation and assumption list.
\begin{itemize}%
\item Let $K$ be a totally real number field.%
\item Let $K_{\infty}$ be a possibly infinite totally real  extension of $K$.
\item Let $F$ be a finite extension of $K$ such that $F \cap K_{\infty} =K$. ($F$ can be equal to $K$.)%
\item Let ${\tt E}$ denote an elliptic curve defined over $F$ with $\rank {\tt E}(F) >0$ and
$i=[{\tt E}(FK_{\infty}):{\tt E}(F)] <\infty$.%

\item Let $O$ denote the identity element of the Mordell-Weil group of ${\tt E}$.
\end{itemize}
\end{notationassumption}

We will separate our Diophantine definition into two parts.
\begin{lemma}%
\label{le:Part1.1}%
Under assumptions in \ref{not:integers} there exists a set {\bf A} contained in $O_{FK_{\infty}}$
and Diophantine over $O_{FK_{\infty}}$, such that if $x \in$ {\bf A} then $\mte{n}_{F(x)}(x)$ can
be considered as a divisor of $F$. Further, if $x$ is a square of a rational integer, then $x \in$
{\bf A}, and all the equations comprising the Diophantine definition of {\bf A} can be satisfied
with variables (except for the variables representing the points on ${\tt E}$) ranging over $O_F$.
\end{lemma}%
\begin{proof}%
Consider the following equations, where $x,  u_Q, v_Q, w, Z, W, A, a, b \in O_{FK_{\infty}}, x(Q),
x(S) \in FK_{\infty}$.

\begin{equation}
\label{eq:firstpart1.1}
S, Q  \in [i]{\tt E}(FK_{\infty}) \setminus \{O\},
\end{equation}
\begin{equation}
\label{eq:Q1.1}
x(Q)=\frac{u_Q}{v_Q},
\end{equation}

\begin{equation}%
\label{eq:Q1.2}
Xb +Yv_Q=1
\end{equation}%
  \begin{equation}%
\label{eq:Q1.3}
Zx +Wu_Q=1
\end{equation}%

\begin{equation}%
\label{eq:Q1.31}
\frac{x(Q)}{x(S)}=\frac{a}{b},
\end{equation}%

\begin{equation}%
\label{eq:cong1.1}
u_Q(xb -a)^2 = v_Qw,
\end{equation}
\begin{equation}%
\label{eq:divisibility1.1}%
v_Q=x^4A.
\end{equation}%

Indeed suppose that equations (\ref{eq:firstpart1.1})--(\ref{eq:divisibility1.1}) are satisfied with all the
variables ranging over the sets described above. Let $M = K(x)$. From equation (\ref{eq:firstpart1.1}) we know
that $x(Q) \in F$.  Since $x(Q) \in F$,  by Lemma \ref{le:evenorder}, we can conclude that $\mte{d}_{FM}(x(Q))$
can be viewed as a second power of a divisor of $F$.  Note that  $\mte{n}_{FM}(v_Q) = \mte{d}_{FM}(x(Q))\mte{W}$,
where $\mte{W}$ is an integral divisor of $FM$ and $\mte{W}$ divides $\mte{n}_{FM}(u_Q)$ in the integral divisor
semigroup of $FM$.  Since $(u_Q,x) =1$ from (\ref{eq:Q1.3}), we must conclude that $\mte{n}_{FM}(x^4)$ divides
$\mte{d}_{FM}(x(Q))$  in the integral divisor semigroup of $FM$. From (\ref{eq:cong1.1}) we see that
$\displaystyle \mte{n}_{FM}\left (\frac{v_Q}{u_Q} \right )$, and therefore $\mte{d}_{FM}(x(Q))$,
divide $\mte{n}_{FM}(xb-a)^2$. Since $b$ and $v_Q$ are relatively prime by (\ref{eq:Q1.2}), we
deduce that $\sqrt{\mte{d}_{FM}(x(Q)})$ also divides $\displaystyle \mte{n}_{FM}\left
(x-\frac{a}{b}\right )$.     Next note that  $\displaystyle \frac{a}{b} \in F$ and  by Lemma
\ref{le:modify} we have the desired conclusion.

Suppose now that $x =m^2$ where $m \in \Z_{>0}$. Let $Q \in [i]{\tt E}(FK_{\infty}) $ be of infinite order and such that
$m^4$ divides $\mte{d}_F(x(Q))$.  Such an $Q$ exists by Lemma \ref{le:Po1}.   Let $S = [m]Q$.  By Lemma
\ref{le:ratio} we have that $\mte{d}_{FM}(x(Q))$ is relatively prime to $\displaystyle
\mte{d}_{FM}\left(\frac{x(Q)}{x(S)}\right )$.  Thus by Lemma \ref{le:denom} we can write
$\displaystyle x(Q) =\frac{u_Q}{v_Q}, \frac{x(Q)}{x(S)}=\frac{a}{b}$,  where $u_Q, v_Q, a, b \in
O_F, (v_Q,b) =1, (\mte{n}_F(u_Q), \mte{d}_F(x(Q)))=1$. Therefore, since $\nn_F(x)$ divides
$\mte{d}_F(x(Q))$, we also have that $(x,u_Q)=1$.      Further by Lemma \ref{le:Po2}, we also have
$\mte{d}_F(x(Q))$ divides $\displaystyle \mte{n}_F\left (\frac{x(Q)}{x(S)}-m^2\right )^2$ as
integral divisors of $F$.   As above $\mte{n}_F(v_Q)= \mte{d}_F(x(Q))\mte{W}$, where $\mte{W}$ is
an integral divisor dividing $\mte{n}_F(u_Q)$.  Therefore $\mte{n}_F(v_Q)$ divides $\displaystyle
\mte{n}_F(u_Q)\mte{n}_F\left (\frac{x(Q)}{x(S)}-m^2\right )^2$. Thus all the equations above can be
satisfied.

Finally we note that all the equations above can be rewritten so that the variables range over $O_{FK_{\infty}}$
only.
\end{proof}%

We now proceed to the second part  Diophantine definition.
\begin{proposition}%
\label{prop:thenbelow1.1}%
Consider the following equations and conditions, where $x, z \in O_{K_{\infty}};
x(Q),  x(P_j) \in FK_{\infty}; u_Q, v_Q, y_j, a_j, b_j,  X_j, Y_j, c_j, d_j, U_j, V_j \in
O_{FK_{\infty}};  j=0,1, 2$.
\begin{equation}%
\label{eq:firstpart1.2}
z \in \mbox{\bf A},
\end{equation}

\begin{equation}%
\label{eq:x1.1}%
x_j=(x+j)^2 \in \mbox{\bf A}, j=0,1,2
\end{equation}%

\begin{equation}%
\label{eq:sigma1.1}%
\sigma(x_j) \geq 1, j=0,1, 2
\end{equation}%
for all $\sigma$, embeddings of $K_{\infty}$ into $\R$,
\begin{equation}
\label{eq:first1.1}
Q, P_0,P_1, P_2 \in [i]E(FK_{\infty}) \setminus \{O\},
\end{equation}

\begin{equation}
\label{eq:classnumber1.1}
\sigma(z) > \max \{\sigma(x_0),\sigma(x_1),\sigma(x_2)\}
\end{equation}
for all $\sigma$, embeddings of $K_{\infty}$ into $\R$,
\begin{equation}
\label{eq:newz1.1}
z_j=x_jz, j=0,1,2
\end{equation}
\begin{equation}
\label{eq:P_21.1}
x(Q) = \frac{u_Q}{v_Q},
\end{equation}
\begin{equation}    %
\label{eq:P_21.2}
X_jv_Q + Y_jb_j=1
\end{equation}
\begin{equation}%
\label{eq:P21.3}%
U_jz_j + V_ju_Q=1
\end{equation}%

\begin{equation}
\label{eq:relprime1.1}
\frac{a_j}{b_j}=\frac{x(Q)}{x(P_j)}, j=0,1,2
\end{equation}

\begin{equation}
\label{eq:divide1.1}
v_Q =p^2z_j^8y_j,
\end{equation}

\begin{equation}%
\label{eq:last1.1}%
u_Q(x_jb_j-a_j)^2=c_jv_Q, j=0,1,2
\end{equation}%
We claim that if these equations are satisfied with variables in the sets as indicated above, then $x \in
O_K$. Conversely, if $x \in \Z_{>0}$ then all the equations can be satisfied with $z, y, y_j, x_j \in
O_K; x(Q), x(P_j) \in F; a, b, u_j \in O_{FK}; j=0,1,2$.
\end{proposition}%

\begin{proof}%
Let $M$ be the Galois closure of $K(x,z)$ over $\Q$. From equation (\ref{eq:x1.1}) we conclude that $x_j \in
M \subset FM, j=0, 1, 2$. Observe also that from (\ref{eq:firstpart1.2}) and (\ref{eq:x1.1}) we have
that $z$ and $x_j$ are elements of {\bf A}, and thus by definition of {\bf A}, we have that
$\mte{n}_{FM}(z_j)=\mte{n}_{FM}(x_jz)$ can be considered as a divisor of $F$. Further, from
(\ref{eq:sigma1.1}) and (\ref{eq:classnumber1.1}) we have that
\begin{equation}%
\label{eq:bigger1.1}
1 \leq \sigma(x_j) < \sigma(z_j)
\end{equation}%
for any embedding $\sigma: FM \longrightarrow \C$. Next using the fact that $b_j$ and $v_Q$ are
relatively prime by equation (\ref{eq:P_21.2}), as in Lemma \ref{le:Part1.1} we conclude that
$\displaystyle {\mathfrak d}_{FM}\left (\frac{x(Q)}{x(P_j)}\right )$ divides
$\displaystyle \mte{n}_{FM} \left (x_j -\frac{a_j}{b_j}\right )^2$ for  $j=0,1,2$ in the integral
divisor semigroup of $FM$.   Next, again as in Lemma \ref{le:Part1.1}, since
$\mte{n}_{FM}(p^2z_j^8)$ divides $\mte{d}_{FM}(x(Q))$ in the integral divisor semigroup of $FM$
by equation (\ref{eq:divide1.1}) and equation (\ref{eq:P21.3}), we also have that
$\mte{n}_{FM}(pz_j^4)$ divides  $\displaystyle \mte{n}_{FM} \left (x_j -\frac{a_j}{b_j}\right )$
for  $j=0,1,2$.
Now by Proposition \ref{prop:usebounds}, we can conclude that $x_j \in F$. But $x_j \in M$ and $M\cap F =K$. Thus
for all $j=0,1,2$ we have that $x_j \in K$. Now by Lemma 5.1 of \cite{Sh1}, we can
conclude that $x \in K$.

Suppose now that $x \in \Z_{>0}$. Then $x_j$ is a non-zero square of a rational integer. Next
choose $z$ to be a square of a rational integer satisfying (\ref{eq:classnumber1.1}).  From this
point on the argument proceeds in the same fashion as in Lemma \ref{le:Part1.1}.
\end{proof}%
The last lemma is all we need for the proof of the following theorem.
\begin{theorem}%
\label{thm:mainint}
Let $K$ be a totally real field.  Let $K_{\infty}$ a totally real possibly infinite algebraic extension of $K$.
Let $F$ be a finite extension of $K$ such that for some elliptic curve ${\tt E}$ defined over $F$
and of a positive rank over $F$ we have that $[{\tt E}(FK_{\infty}):{\tt E}(F)] < \infty$ and
$K_{\infty} \cap F=K$.  Then $O_K$ and $\Z$ have a Diophantine definition over $O_{K_{\infty}}$ and
Hilbert's Tenth Problem is not solvable over $O_{K_{\infty}}$.
\end{theorem}%

We now state a corollary whose proof follows  from our definition of definability of integrality at a finite set
of primes in infinite algebraic extensions.
\begin{corollary}%
\label{cor:smallrings}%
Let $\calC_K \subset \calP(K)$ be a set of primes of $K$ such that integrality is definable at primes of $\calC_K$
in $K_{\infty}$. Then $O_{K,\calC_K}, O_K$ and $\Z$ are existentially definable over
$O_{K_{\infty},\calC_{K_{\infty}}}$  and therefore HTP is not solvable over $O_{K_{\infty},\calC_{K_{\infty}}}$.
\end{corollary}%

Theorem \ref{thm:mainint} and Corollary can also be stated in a different way avoiding reference to any number
fields.  (See Main Theorem A.)
\begin{theorem}%
\label{thm:mainint2}%
Let $K_{\infty}$ be a totally real possibly infinite algebraic extension of $\Q$. Let $U_{\infty}$ be a finite
extension of $K_{\infty}$ such that there exists an elliptic curve ${\tt E}$ defined over $U_{\infty}$ with
${\tt E}(U_{\infty})$ finitely generated and of a positive rank. Then $\Z$ is existentially
definable and HTP is unsolvable over the ring of integers of $K_{\infty}$.  Further, $\Z$ is
existentially definable and HTP is unsolvable over any small subring of $K_{\infty}$ where
integrality is definable at all the primes allowed in the denominator of the divisors  of elements
of the ring.
\end{theorem}%

\section{Diophantine Definitions of Some Big and Arbitrary Small Subrings of Rational Numbers  for the
Totally Real Case.}%
\label{sec:real}
In this section we will consider large rings and arbitrary small rings.  However in order to be able to manage
large sets of primes in the denominator it is necessary to make additional assumptions on the nature of the
infinite extension $K_{\infty}$. These extra assumptions are listed below.
\begin{notationassumption}%
\label{not:biggerrings}%
In what follows we add the following to the list of notation and assumptions.
\begin{itemize}%
\item $K_{\infty}$ is normal over $K$.
\item Let $E$ be a number field satisfying  the following conditions:
\begin{itemize}
\item $n_E=[E:\Q]=[EK:K]$.%
\item For any number field $M \subset K_{\infty}$ and such that $K \subset M$ we have that $([M:K], n_E)=1$.
\end{itemize}%
\item Let $\mu_E \in O_E$ be a generator of $E$ over $\Q$. Assume also that $\mu_E$ is an integral
unit. Let $R(T)
\in \Z[T]$ be the monic irreducible polynomial of an integral generator $\mu_E$ of $E$ over $\Q$. %
\item Let $\calV_K \subset \calP(K)$ be a set of primes of $K$ without relative degree 1 factors in $E$. %
\item Let $\calS_K$ be a finite set of primes of $K$. \item Let $\calW_K = \calV_K \cup \calS_K$.%
\item Let $\calE_K$ consist of all the primes $\pp_K$ of $\calW_K$ such that either $\pp_K$ divides the
discriminant
of $R(T)$ or $\pp_K$ has a relative degree 1 factor in the extension $EK/K$.%
\item For any $C >0$, let $A(C)\geq 2$ be an integer such that for any real $t > A(C)$ we have that $R(t) >C$.%
\item Let $N_0=0, \ldots,N_{2n_E}$ be positive integers selected so that the set of polynomials
$\{R(A(1) + x+N_j)^2\}$ is linearly independent. Such a set of positive integers exists by Lemma
12.1 of \cite{Sh36}.)
\item Let $p$ be a rational prime without any factors in $\calW_K$.%
\item For any number field $U$, any set $\calD_U \subset \calP(U)$ and any $x \in U$ let
\[%
\mte{n}_{U,\calD_U}(x)= \prod_{\pp \in \calP(U)\setminus \calD_U, \ord_{\pp}x >0}\pp^{\ord_{\pp}x}
\]%
\end{itemize}%
\end{notationassumption}%
We start with some observations concerning our prime sets.
\begin{lemma}%
\be
\item $\calE_K$ is a finite set of primes.
\item All the primes of $\overline \calE_K$ are $K_{\infty}$-boundable.
\ee
\end{lemma}%
\begin{proof}%
\be%
\item Only finitely many can divide the discriminant of $R(T)$.
\item By assumption the extension $K_{\infty}/K$ satisfies the requirement of Corollary \ref{cor:deg}.  Thus, any
finite set of primes of $K$ is $K_{\infty}$-boundable.%
  \ee
\end{proof}%
Next we note that from our assumptions on the degree of subextensions of $K_{\infty}$  and Lemma
\ref{le:nodegonefact} we can obtain the following lemma.
\begin{lemma}%
Let $M \subset K_{\infty}$ be a number field containing $K$.  Let $\pp_M$ be a prime lying above a prime of $K$
without relative degree one factors in the extension $EK/K$.  Then $\pp_M$ does not have a relative
degree one factor in the extension $EM/M$.
\end{lemma}%

We now proceed to the big ring versions of Lemma \ref{le:Part1.1} and Proposition \ref{prop:thenbelow1.1}.
\begin{lemma}%
\label{le:Part1}%
The exist a set {\bf A} contained in $O_{FK_{\infty},\calW_{K_{\infty}}}$ and Diophantine over
$O_{FK_{\infty},\calW_{FK_{\infty}}}$,  such that if $\ord_{\pp}x \leq 0$ for all $\pp$ in $F(x)$ lying above
primes in $\overline \calW_F$, and  $x \in$ {\bf A}, then $\mte{n}_{F(x)}(x)$ can be considered as a divisor of
$F$ (composed solely of factors of $F$-primes outside $\overline \calW_F$).  Further, if $x $ is a square of a
rational integer  then $x \in$ {\bf A} and all the equations comprising the Diophantine definition can be
satisfied with variables ranging over $O_{F,\calW_F}$.
\end{lemma}%
\begin{proof}%
Consider the following equations, where $x, y \in O_{FK_{\infty},\calW_{FK_{\infty}}}, u_Q, v_Q, w, A, B, C \in
O_{FK_{\infty},\calW_{FK_{\infty}}}$, $x(Q), x(R) \in FK_{\infty}$ and $\ord_{\pp} x \leq 0$ for all
$\pp$ in $F(x)$ lying above primes in $\overline \calW_F$.

\begin{equation}
\label{eq:firstpart1}
S, Q,  \in [i]{\tt E}(FK_{\infty}) \setminus \{O\},
\end{equation}%

\begin{equation}
\label{eq:Q}
x(Q)=\frac{u_Q}{v_Q}, v_Q \not=0,
\end{equation}

\begin{equation}%
\label{eq:Q1.4}
Xb +Yv_Q=1
\end{equation}%

\begin{equation}%
\label{eq:Q1.4.1}%
\frac{x(Q)}{x(S)}  =\frac{a}{b}
\end{equation}%

  \begin{equation}%
\label{eq:Q1.5}
Zx +Wu_Q=1
\end{equation}%

\begin{equation}
\label{eq:cong}
u_Q(bx - a )^2 = v_Qw,
\end{equation}
\begin{equation}%
\label{eq:divisibility}%
v_Q=x^4A.
\end{equation}%
We claim that if for some $x \in O_{FK_{\infty},\calW_{FK_{\infty}}}$  these equations are satisfied with all the
variable as indicated above, then $x$ satisfies the requirements for the membership in {\bf A} as described in the
statement of the lemma, and if $x$ is a square of an integer, then all the equations can be satisfied with all the
variables ranging over  $O_{F,\calW_F}$.

Indeed suppose that equations (\ref{eq:firstpart1})--(\ref{eq:divisibility}) are satisfied with all the variables
ranging over the sets described above.  Let $M = K(x)$.  Then by assumption we have that for all $\pp \in
\overline \calW_{FM}$ it is the case that $\ord_{\pp}x \leq 0$.  Next let $\mte{d}_{FM}(x(Q))=
\mte{N}_1\mte{N}_2$, where $\mte{N}_1, \mte{N}_2$ are integral divisors of $FM$,  all the primes  occurring in
$\mte{N}_1$ are outside  $\overline \calW_{FM}$ and all the primes  occurring in $\mte{N}_2$ are in
$\overline \calW_{FM}$.  Since $x(Q) \in F$ and $\overline \calW_{FM}$ is, by definition,  closed under conjugation
over $F$, by Lemma \ref{le:evenorder}, we can conclude that $\mte{N}_1, \mte{N}_2$ can be both viewed as second
powers of divisors of $F$.  Next write
\begin{equation}%
\mte{n}_{FM}(v_Q) = \mte{N}_1\mte{A}\mte{B}, \mte{n}_{FM}(u_Q)=\mte{C}\mte{A}\mte{D},
\end{equation}%
where $\mte{N}_1, \mte{A}, \mte{B}, \mte{C}. \mte{D}$ are integral divisors of $FM$,  $\mte{N}_1, \mte{A}, \mte{B}$
are pairwise relatively prime,  $\mte{C},\mte{A}, \mte{D}$ are pairwise relatively prime,  $\mte{C},\mte{A}$ are
composed of primes outside $\overline \calW_{FM}$ only, while $\mte{B}, \mte{D}$ are composed of
primes in $\overline \calW_{FM}$.   From equation (\ref{eq:Q1.5}) we can conclude that
$(\mte{n}_{FM}(x),\mte{A})=1$, and since $x$ does not have a positive order at any prime of
$\overline \calW_{FM}$, we can conclude from equation (\ref{eq:divisibility}) that
$\mte{n}_{FM}(x)$ divides $\mte{N}_1$ in the integral divisor semigroup of $FM$. From
(\ref{eq:Q1.4}) we know that $\mte{n}_{FM}(b)$ is relatively prime to $\mte{N}_1$ and we have
already established that $\mte{N}_1$ is relatively prime ot $\mte{n}_{FM}(u_Q)$.    So if  we let
$\displaystyle \mte{M} = \mte{n}_{FM}\left (x - \frac{a}{b} \right )$.  Then from (\ref{eq:cong})
we conclude that $\mte{N}_1$ divides $\mte{M}$ in the semigroup of integral divisors of $FM$.
Therefore, by Lemma \ref{le:modify} we have the desired conclusion.

Suppose now that $x =m^2$ where $m \in \Z_{>0}$.   Then we can again proceed as in Lemma \ref{le:Part1.1}.

Finally we note that all the equations above can be rewritten so that the variables range over
$O_{F,\calW_F}$ only.
\end{proof}%
We now proceed to the part two of the big ring Diophantine definition.
\begin{proposition}%
\label{prop:thenbelow}%
Consider the following equations and conditions, where $x, z, z_0, x_j \in
O_{K_{\infty},\calW_{K_{\infty}}}; x(Q), x(P_j) \in FK_{\infty}; a_j, b_j, c_j,  X_j, Y_j, Z_j,
W_j, w_j \in O_{FK_{\infty},\calW_{FK_{\infty}}};  j=0, \ldots, 2n_E$.
\begin{equation}%
\label{eq:firstpart}
z = R(z_0) \in \mbox{\bf A},
\end{equation}

\begin{equation}%
\label{eq:x}%
x_j=(R(A(1) + x +N_j))^2 \in \mbox{\bf A}, j=0,\ldots,2n_E,
\end{equation}%

\begin{equation}%
\label{eq:sigma}%
 \sigma(x_j) \geq 1
\end{equation}%
for all $\sigma: K_{\infty} \rightarrow \R$,
\begin{equation}
\label{eq:first}
 Q, P_1,\ldots, P_{2n_E} \in [i]{\tt E}(FK_{\infty}) \setminus \{O\},
\end{equation}

\begin{equation}
\label{eq:classnumber}
\sigma(z) \geq \max\{\sigma(x_1),\ldots,\sigma(x_{2n_E})\}
\end{equation}
for all $\sigma$, embeddings of $K_{\infty}$ into $\R$,%

\begin{equation}
\label{eq:newz}
z_j=x_jz
\end{equation}  %

\begin{equation}
\label{eq:P_2}
x(Q) = \frac{u_Q}{v_Q},
\end{equation}

\begin{equation}%
\label{eq:Q1.4.0.1}
X_jb_j +Y_jv_Q=1
\end{equation}%

\begin{equation}%
\label{eq:Q1.4.1.1}%
\frac{x(Q)}{x(P_j)}  =\frac{a_j}{b_j}
\end{equation}%

  \begin{equation}%
\label{eq:Q1.5.1.1}
Z_jz_j +W_ju_Q=1
\end{equation}%

\begin{equation}
\label{eq:divide}
v_Q =p^2z_j^8w_j
\end{equation}

\begin{equation}%
\label{eq:last}
u_Q  (b_jx_j -a_j )^2 = v_Qc_j.
\end{equation}%
We claim that if these equations are satisfied with variables in the sets as indicated above, then $x \in
O_{K,\calW_K}$. Conversely, if $x \in \Z_{>0}$ then all the equations can be satisfied with $z_0,
x \in O_{K,\calW_K}; x(Q), x(P_j) \in F; a_j, b_j, c_j, w_j, X_j, Y_j, Z_j, W_j \in
O_{F,\calW_{F}}; j=0, \ldots, 2n_E$.
\end{proposition}%

\begin{proof}%
Let $M$ be the Galois closure of $K(x,z_0)$ over $\Q$. Next from  equation (\ref{eq:x}) we conclude that $x_j
\in M \subset MF, j=0,\ldots, 2n_E$ and by Lemma \ref{le:negorder},
\begin{equation}%
\label{eq:nonzero}%
\forall \pp \in \overline \calW_M: \ord_{\pp}x_j \leq 0.
\end{equation}%
Similarly, we have from (\ref{eq:firstpart}) that
\begin{equation}%
\label{eq:nonzeroz}%
\forall\pp \in \overline \calW_M: \ord_{\pp}z \leq 0.%
\end{equation}%
Observe also that from (\ref{eq:x}) and (\ref{eq:firstpart}) we deduce that $z$ and $x_j$ are elements of {\bf A},
and thus by definition of {\bf A}, we have that $\mte{n}_{MF}(z_j)=\mte{n}_{MF}(x_jz)$ can be considered as a
divisor of $F$. Additionally from (\ref{eq:nonzero}) and (\ref{eq:nonzeroz}) we know that all the primes occurring
in $\mte{n}_{MF}(z_j)$ are outside $\overline \calW_{MF}$.  Further from (\ref{eq:sigma}) and
(\ref{eq:classnumber}) we now have that
\begin{equation}%
\label{eq:bigger}
1 \leq \sigma(x_j) \leq \sigma(z_j)
\end{equation}%
for any embedding $\sigma: MF \longrightarrow \C$. Next using equations (\ref{eq:divide}) and
(\ref{eq:Q1.5.1.1}), by an argument analogous to the one used in Lemma \ref{le:Part1.1}, we observe
that $\mte{n}_{MF}(pz_j^4)$ divides $\displaystyle \mte{n}_{MF}(x_j -\frac{a_j}{b_j})$ for all
$j=0,\ldots,2n_E$.

Now by Proposition \ref{prop:usebounds}, we can conclude that $x_j \in F$.  But $x_j \in M$ and
$M\cap F =K$. Thus for all $j=0,\ldots,2n_E$ we have that $x_j \in K$.  Further, by our assumption
on $N_0,\ldots, N_{2n_E}$ and by Lemma 5.1 of \cite{Sh1}, we can conclude that $x \in K$. \\

Suppose now that $x \in \Z_{>0}$. Then $x_j$ is a non-zero square of a rational integer. From this
point on  the argument proceeds in the same fashion as in Lemma \ref{le:Part1.1}.
\end{proof}%

We are now ready for the main theorem.%
\begin{theorem}%
\label{thm:1}
 There exists a polynomial equation $P(x,\bar{t}) \in O_K[x,\bar{t}]$ such that the following statements are true.%
\be%
\item For any $x \in O_{K_{\infty},\calW_{K_{\infty}}}$, if $P(x,\bar{t}) =0$ for some
$\bar{t}=(t_1,\ldots,t_m)$ with $t_i \in O_{K_{\infty},\calW_{K_{\infty}}}$, then $x \in O_{K,\calW_K}$.%
\item If $x \in \Z$ there exists $\bar{t}=(t_1,\ldots,t_m)$ with $t_i \in O_{K, \calW_K}$ such that
$P(x,\bar{t}) =0$ 
\item $O_{K_{\infty},\calW_{K_{\infty}}} \cap K = O_{K,\calW_K}$ has a Diophantine definition over
$O_{K_{\infty},\calW_{K_{\infty}}}$.%
\ee
\end{theorem}
\begin{proof}%
The proof of the theorem pretty much follows from Lemma \ref{le:Part1} and Proposition
\ref{prop:thenbelow}. We just need to remind the reader that (\ref{eq:sigma}) can be rewritten in a
polynomial form by Proposition \ref{prop:realembed} and we can make sure that all the variables
range over $O_{K_{\infty},\calW_{K_{\infty}}}$ as opposed to $FK_{\infty}$ or
$O_{FK_{\infty},\calW_{FK_{\infty}}}$. This can be done using Proposition 2.6 of \cite{Sh36}. (For
more extensive discussion of ``rewriting'' issues the reader can see the section on coordinate
polynomials in the Appendix B of \cite{Sh34}.)
\end{proof}%
To get down to $\Q$ we need additional notation and assumptions.
\begin{notationassumption}%
\label{not:bigringtoreal}
We add the following to our notation and assumption list.
\begin{itemize}
\item Let $K^{Gal}$ be the Galois closure of $K$ over $\Q$.%
\item Assume $E/\Q$ is cyclic of prime degree and $n_E >[K^{Gal}:\Q]$.%
\item Using Corollary 7.6.1 of \cite{Sh34} and Proposition \ref{prop:finmany} we know that for some set of
$K$-primes $\calT_K$ such that $\calV_K\subset \calT_K$ and $\calT_K \setminus \calV_K$ is a finite set, we have
that $O_{K,\calT_K}\cap \Q$ has a Diophantine definition over $O_{K,\calT_K}$. From Proposition \ref{prop:finmany}
it also follows that we can add finitely many primes to $\calT_K$ without changing the situation. Thus for the
results below it is enough to assume that $\calS_K$ contains all the primes of $\calT_K \setminus \calV_K$. \item
Let $\calR_K = \hat \calV_K \cup \calS_K$.
\end{itemize}%
\end{notationassumption}%

Given the additional notation and assumptions above we now have the following corollary.
\begin{corollary}%
\label{cor:downreal}%
\be%
\item $O_{K_{\infty},\calW_{K_{\infty}}} \cap \Q$ has a Diophantine definition over
$O_{K_{\infty},\calW_{K_{\infty}}}$.%
\item For any archimedean and non-archimedean topology of $K_{\infty}$ we can choose $\calS_K$ so
that  $O_{K_{\infty},\calW_{K_{\infty}}}$ has an infinite Diophantine subset which is discrete in this topology of the field.%
\item $O_{K,\calW_K}$ is contained in a set with a (natural or Dirichlet) density is $1-\frac{1}{[E:\Q]}$.%
\item $O_{K_{\infty},\calR_{K_{\infty}}} \cap \Q$ has a Diophantine definition over
$O_{K_{\infty},\calR_{K_{\infty}}}$ and therefore HTP is unsolvable over
$O_{K_{\infty},\calR_{K_{\infty}}}$%
\item $O_{K, \calR_K}$ is contained in a set with a (natural or Dirichlet) density less or equal to
$1-\frac{1}{[K:\Q]}-\frac{1}{[E:\Q]}$.
 \ee%
\end{corollary}%
\begin{proof}%
\be%
\item This assertion follows directly from Theorem \ref{thm:1}, Proposition \ref{prop:finmany} and Corollary 7.6.1
of \cite{Sh34}.%
\item  This part of the corollary follows from Section 3 of \cite{Sh21} and Section 2 of \cite{PS}. %
\item This  statement  follows from the fact that the set of $K$ primes inert in the extension $EK/K$ has
(Dirichlet or natural) density $1-\frac{1}{[E:\Q]}$. %
\item This assertion is true because by construction of $\calR_K$ we have that $O_{K_{\infty},\calR_{K_{\infty}}}
\cap \Q$ is a ``small'' ring, i.e. a ring of the form $O_{\Q,\calT_{\Q}}$, where $\calT_{\Q}$ is a finite, possibly
empty set of rational primes. Thus, by Proposition \ref{prop:finmany}, $\Z$ has a Diophantine definition over
$\calT_{\Q}$ and therefore  over $O_{K_{\infty},\calR_{K_{\infty}}}$. Hence HTP is unsolvable over this ring.
\item  This part of the corollary follows from a standard density calculation (see for example Section B.5 of
\cite{Sh34}).
\ee
\end{proof}%
We restate our results in the following two formulations.
\begin{theorem}%
\label{thm:mainnonint0}%
Let $K_{\infty}$ be a totally real possibly infinite algebraic extension of $\Q$ normal over some
number field and such that for some rational prime number $p$ we have that $i_{K_{\infty}}(p)=0$.
Let $U_{\infty}$ be a finite extension of $K_{\infty}$ such that there exists an elliptic curve
${\tt E}$ defined over $U_{\infty}$ with ${\tt E}(U_{\infty})$ finitely generated and of a positive
rank. Then $K_{\infty}$ contains a large subring $R$ such that $Z$ is definable over $R$ and HTP is
unsolvable over $R$.
\end{theorem}%

\begin{theorem}%
\label{thm:mainnonint}%
Let $K_{\infty}$ be a totally real possibly infinite algebraic extension of $\Q$ normal over some number field and
such that there exists a  non-repeating sequence of  rational prime numbers $\{p_i\}$ with the
property that $i_{K_{\infty}}(p_i)<\infty$. Let $U_{\infty}$ be a finite extension of $K_{\infty}$
such that there exists an elliptic curve ${\tt E}$ defined over $U_{\infty}$ with ${\tt
E}(U_{\infty})$ finitely generated.  Then for any $\varepsilon  >0$ we have that $K_{\infty}$
contains a large subring $R$
satisfying the following conditions: %
\be%
\item There exists a number field $K \subset K_{\infty}$ and a set $\calW_K \subset \calP(K)$ of (Dirichlet or
natural) density greater than $1-\varepsilon$ such that $R=O_{K_{\infty},\calW_{K_{\infty}}}$.%
\item $\Z$ is definable over $R$ and HTP is unsolvable over $R$.%
\ee%
\end{theorem}%
Before stating the theorem concerning arbitrary small rings we change assumptions again.
\begin{notationassumption}
We will now remove conditions imposed on $\calS_K$ in Notation and Assumptions \ref{not:bigringtoreal}.  That is in
what follows we again assume that $\calS_K$ is an arbitrary finite set of primes of $K$.
\end{notationassumption}
First as a consequence of  Theorem \ref{thm:1} we have the following corollary.
\begin{corollary}%
$\Z$, $O_{K_{\infty},\calS_{K_{\infty}}} \cap \Q$ and $O_{K,\calS_K}$ have  Diophantine definitions over
$O_{K_{\infty},\calS_{K_{\infty}}}$.  Thus HTP is not solvable over  $O_{K_{\infty},\calS_{K_{\infty}}}$.
\end{corollary}%

This result can also be restated in the following form.
\begin{theorem}%
\label{thm:smallring}%
Let $K_{\infty}$ be a totally real possibly infinite algebraic extension of $\Q$ normal over some number field and
such that for some  rational prime number $p$   we have that $i_{K_{\infty}}(p)=0$. Let $U_{\infty}$
be a finite extension of $K_{\infty}$ such that there exists an elliptic curve ${\tt E}$ defined over
$U_{\infty}$ with ${\tt E}(U_{\infty})$ finitely generated. Then for any small subring $R$ of
$K_{\infty}$ we have that $\Z$ is definable over $R$ and HTP is unsolvable over $R$.
\end{theorem}%

\section{Diophantine Definition of Rational Integers for Extensions of Degree 2 of Totally
Real Fields.}%
\setcounter{equation}{0}
Most of the work necessary for treating the case of extensions of degree 2 of the totally real fields discussed
above has been done in \cite{Sh36}. However we will revisit some of the results because in the current case the
presentation can be simplified and the results concerning rings of integers can be made slightly more general. We
start, as usual with notation and assumptions.

\begin{notationassumption}
\label{not:ext2integers}
In this section we make the following changes and additions to the notation and assumption list.
\begin{itemize}%
\item For the case of integers we remove all the assumptions on $K_{\infty}$
besides the fact that it is an algebraic possibly infinite extension of $\Q$ and $K$ contains a totally positive
element (that is an element all of whose $\Q$-conjugates are positive) which is not a square in $K_{\infty}$.%
\item Let $G$ be an extension of degree 2 of $K$. Let $\alpha_G$ be a generator of $G$ over $K$ with
$\alpha_G^2=a_G \in O_K$. Assume $G$ is not a totally real field.%
\item Assume $[GK_{\infty}:K_{\infty}]=2$.%
\item Let $d_H \in O_K$ be such that it is not a square in $K_{\infty}$ and all the conjugates of $d_H$ over $\Q$
are greater than 1 in absolute value.%
\item Assume that for any embedding $\sigma: K \longrightarrow \C$ we have that $\sigma(d_H) >0$ if
and only if $\sigma(a_G)
<0$.%
\item Let $\delta_H \in \C$ be a square root of $d_H$. \item Let $H=K(\delta_H)$. \item Given a
number field $T$ we let $U_T$ denote the group of integral units of $T$, let $s_T$ denote the
number of non-real embeddings of $T$ into $\C$, and let $r_T$ denote the number of real
embeddings of $T$ into $\C$.%
\item Let $M$ as before be a number field contained in $K_{\infty}$ and containing $K$.

\end{itemize}%
\end{notationassumption}
\begin{remark}
$d_H$ as described above exists by the Strong Approximation Theorem and our assumptions.
\end{remark}
We start our discussions of technical details with a proposition due to Denef and Lipshitz in \cite{Den2}.
\begin{proposition}%
\label{prop:deneflipshitz}%
Let $M \subset K_{\infty}$ be a number field with $K \subseteq M$.  Let
\[%
A_{GM}= \{\varepsilon \in U_{GHM}: {\mathbf N}_{GHM/GM}(\varepsilon)=1\}
\]%
\[%
A_{M}= \{\varepsilon \in U_{HM}: {\mathbf N}_{HM/M}(\varepsilon)=1\}
\]%
Then $A_{GM}$ and $A_M$ are multiplicative groups of equal rank.
\end{proposition}%

From this proposition we derive the following corollary.
\begin{corollary}%
\label{cor:4power}
Let
\[%
B_{GM}= \{x +\delta_H y: x, y \in O_G, x+\delta_Hy  \in A_{GM}\},
\]%
\[%
B_{M}= \{x +\delta_H y: x, y \in O_M, x+\delta_Hy  \in A_{M}\}.
\]%
Then if  $x +\delta_H y \in B_{GM}$, it follows that $(x + \delta_H y)^4 \in B_M$.  Further, assuming that
$[M:\Q] \geq 2$ we have that $B_M$ contains elements of infinite order.
\end{corollary}%
\begin{proof}%
Let $\sigma_G$ be the generator of $\Gal(GM/M)$   and let $\sigma_{H}$  be a generator of $\Gal(HM/M)$.  We can
extend both elements to $GHM$ by requiring $\sigma_{G}$ to be the identity on $H$ and similarly
$\sigma_H$ to be the identity on $G$.  Then we observe that $\Gal(GHM/M)$ is generated by $\sigma_G$
and $\sigma_H$ and $\sigma_G\sigma_H=\sigma_H\sigma_G$.

Now let $\varepsilon \in A_{GM}$.   Then $\sigma_G(\varepsilon) \in A_{GM}$.  Indeed,    $\varepsilon \in A_{GM}$
if and only if $\varepsilon \sigma_H(\varepsilon)=1$.  Therefore we have the following equality:
\[%
\sigma_G(\varepsilon)\sigma_H(\sigma_G(\varepsilon))=\sigma_G(\varepsilon \sigma_H(\varepsilon))=1.
\]%
Given our definition of $B_{GM} =A_{GM} \cap O_G[\delta_H]$, it also follows that  $\varepsilon \in B_{GM}$ if and
only if $\sigma_G(\varepsilon) \in B_{GM}$.  Further, it is also the case that   $\varepsilon \in
B_{GM}$  if and only if $\varepsilon^{-1} \in B_{GM}$.   Finally it is clear that $B_{GM}$ is a
multiplicative group and from Proposition \ref{prop:deneflipshitz} we have that for any
$\varepsilon \in B_{GM}$ for some $l$ it is the case that $\varepsilon^l \in B_M$ and consequently
$\varepsilon^l = \sigma_G(\varepsilon)^l$.  Thus,
$\displaystyle \frac{\varepsilon}{\sigma_G(\varepsilon)}=\xi_l$, where $\xi_l$ is a root of  unity
which is also an element of $B_{GM}$.  By Lemma \ref{le:rootofunity}, $\xi_l^4=1$.  Thus our first
assertion is true.

Next we observe that given our assumption on $[M:\Q]$ we have that  $A_M$ has elements of infinite order.  Finally,
by Lemma  6.1.4 and Lemma B.4.12 of \cite{Sh34}, we conclude that $B_M$ has elements of infinite order.
\end{proof}%
Next for the convenience of the reader we state two technical lemmas which are simplified versions of
results in \cite{Sh36} concerning bounds.
\begin{lemma}%
\label{le:bound1}%
 Let $x=y_0 + y_1\alpha_G \equiv z \mod \mte{Z}$ in $O_{GHM}$ where $x \in O_{GM},  y_0, y_1 \in O_{M}, z \in
O_{HM},  \mte{Z}$ an integral divisor of $HM$.  Let $Z = {\mathbf N}_{GHM/\Q}(\mte{Z})$.  Then
$\displaystyle \frac{{\mathbf N}_{GHM/\Q}(2\alpha y_1)}{Z}$  is an integer.
\end{lemma}%
(See Lemma 6.6 of \cite{Sh36}. )
\begin{lemma}%
\label{le:bound2}%
Let $x \in GM, x = y_0 + \alpha_Gy_1, y_1, y_2 \in M, z \in GM$. Suppose further that for any
$\sigma$, an embedding of $GM$ into $\R$, we have that
\[%
1 \leq |\sigma(x)| < |\sigma(z)|
\]%
while for any $\tau$, a non-real embeddings of $GM$ into $\C$, we have that
\[%
\tau(z) \geq 1.
\]%
Then $|{\mathbf N}_{GM/\Q}(y_1)| \leq |{\mathbf N}_{GM/\Q}(x)||{\mathbf N}_{GM/\Q}(z)|$.
\end{lemma}%
(See Lemma 6.9 of \cite{Sh36}.)\\%
\begin{lemma}%
\label{le:closeto1}%
Let $\varepsilon \in B_K$. Then for any positive integer $k$ and any $\lambda >0$ there exists a positive integer
$r$ such that for all $\tau : H \rightarrow \C $ with $\tau(H) \not \subset \R$ we have that
\[%
\left|\frac{\tau(\varepsilon^{rk}-1)}{\tau(\varepsilon^r-1)}-k\right|< \lambda.
\]%
\end{lemma}%
(See Lemma 7.2 of \cite{Sh36}.)\\

We are now ready to list the equations comprising a Diophantine definition we seek.
\begin{theorem}%
\label{thm:2integers}
 There exists a polynomial equation $P(x,\bar{t}) \in O_K[x,\bar{t}]$ such that the following statements are true.%
\be%
\item For any $x \in O_{GK_{\infty}}$, if $P(x,\bar{t}) =0$ for some $\bar{t}=(t_1,\ldots,t_m)$
with $t_i \in O_{GK_{\infty}}$, then $x \in O_{K_{\infty}}$.%
\item If $x \in \Z$ there exists $\bar{t}=(t_1,\ldots,t_m)$ with $t_i \in O_K$ such that $P(x,\bar{t}) =0$ 
\ee
\end{theorem}%
\begin{proof}%
The proof will use Proposition 7.3 of \cite{Sh36} as its foundation. However, our notation is a bit different from
the notation used in that proposition. Let $x \in O_{GK_{\infty}}, a_1, a_2, b_1, b_2, c_1, d_1,
\ldots. c_4, d_4, u, v, u_1, v_1,\ldots,u_4, v_4 \in O_{GK_{\infty}}$, $\gamma_1, \ldots,
\gamma_4\in O_{GHK_{\infty}}$, and assume the following conditions and equations are satisfied.

\begin{equation}%
\label{eqdeg2:2i}
u_i^2 - d_Hv_i^2=1, i=1,\ldots,4,
\end{equation}%

\begin{equation}%
\label{eqdeg2:3i}%
\gamma_i = (u_i - \delta_Hv_i)^4, i=1, \ldots,4,
\end{equation}%
\begin{equation}%
\label{eqdeg2:4i}%
\frac{\gamma_{2j}-1}{\gamma_{2j-1}-1} =a_j-\delta_H b_j, j=1,2
\end{equation}%
\begin{equation}
\label{eqdeg2:4.2i}%
\gamma_i=c_i+\delta_H d_i, i=1,\ldots, 4,
\end{equation}%
\begin{equation}%
\label{eqdeg2:7i}%
1 \leq |\sigma(x)| \leq 1 +\sigma(a_1^2-d_Hb_1^2)^2,
\end{equation}%
where $\sigma$ ranges over all embeddings of $GK_{\infty}$ into $\R$,
\begin{equation}%
\label{eqdeg2:9i}%
x-(a_2-\delta_H b_2) =(c_3-1+\delta_H d_3)(u+v\delta_H),
\end{equation}%
\begin{equation}%
\label{eqdeg2:10i}%
6\alpha_Gx(1+(a_1^2-d_Hb_1^2)^2) \big{|} (c_3-1+\delta_H d_3).
\end{equation}%
Then, we claim,  $x \in O_{K_{\infty}}$.

Conversely, we claim that if $x \in \Z_{>0}$ , the conditions and equations above can be satisfied
with $a_1, a_2$, $b_1, b_2, c_1, d_1, \ldots, c_4, d_4$,  $u, v, u_1, v_1, z_1,\ldots,u_4, v_4,
\gamma_i \in O_{HK}, i=1,\ldots,4$.

To prove the first claim, observe the following. Let $M$ be such that $HGM$ contains $\alpha_G, \delta_H, x,
a_1,a_2$, $b_1,b_2, c_1, d_1, \ldots, c_4, d_4, u, v, u_1, v_1, \ldots, u_4, v_4, \gamma_1,\ldots,
\gamma_4 $. Then given our assumptions on the fields under consideration, in the equations above we
can replace $HGK_{\infty}$ by $HGM$, $GK_{\infty}$ by $GM$, and finally $K_{\infty}$ by $M$, while
the equalities and other conditions will continue to be true.

By Corollary \ref{cor:4power}  and equations (\ref{eqdeg2:2i}) and (\ref{eqdeg2:3i}) we have that
$\gamma_1,\ldots,\gamma_4 \in HM$, and consequently from equation (\ref{eqdeg2:4i}) we also have that $a_j^2 -
d_Hb_j^2 \in M$ for $j=1,2$.   From equation (\ref{eqdeg2:7i}) we know that
\[%
1 \leq \sigma(x) \leq 1 + \sigma(a_1^2-d_Hb_1^2)^2
\]%
for all real embeddings $\sigma$ of $GM$. Next observe that
\[%
x - (a_2 - \delta_H b_2) \equiv 0 \mod (\gamma_3 -1),
\]%
and therefore if $x=y_0 + y_1\alpha_G, y_0, y_1 \in M$, by Lemma \ref{le:bound1} we have that
\[%
{\mathbf N}_{HGM/\Q}(2y_1\alpha_G) \equiv 0 \mod  {\mathbf N}_{HGM/\Q}(\gamma_3 -1).
\]%
So either $y_1=0$ or
\[%
{\mathbf N}_{HGM/\Q}(2y_1\alpha_G) \geq  {\mathbf N}_{HGM/\Q}(\gamma_3 -1).
\]%
Observe that for any $\tau: GM \longrightarrow \C \setminus \R$ we have that
$|\tau(1+(a_1^2-d_Hb_1^2)^2)| \geq 1$ since  $\tau(a_1^2-d_Hb_1^2) \in \R$.
 Thus  by
Lemma \ref{le:bound2}, equation (\ref{eqdeg2:7i}) and equation (\ref{eqdeg2:10i}) we have that
\[%
{\mathbf N}_{HGM/\Q}(2\alpha_Gy_1) \leq {\mathbf N}_{HGM/\Q}(2\alpha_G) {\mathbf N}_{HGM/\Q}(x){\mathbf
N}_{GMH/\Q}(1+(a_1^2-d_Hb_1^2)^2) <{\mathbf N}_{GMH/\Q}(\gamma_3 -1).
\]%

Therefore unless $y_1=0$, we have a contradiction.\\

We will now show that assuming that $x > 1$ is a natural number, we can satisfy all the equations
and conditions (\ref{eqdeg2:2i})--(\ref{eqdeg2:10i}). Let $\nu \in B_K$ be a unit of $O_{H}$ which
is not a root of unity.   Such a $\nu$ exists by Corollary \ref{cor:4power}.    Let
$\{\phi_1,\ldots,\phi_{s_H}\}$ be a set containing a representative from every complex-conjugate
pair of non-real conjugates of $\nu$. By Lemma \ref{le:closeto1}, we can find a positive integer $r
\cong 0 \mod 4$ such that for all $i=1,\ldots, s_{H}$ we have that
\[%
\left |\frac{\phi_i^{rx}-1}{\phi_i^r-1}-x\right| < \frac{1}{2},
\]%
and thus,
\[%
\left |\frac{\phi_i^{rx}-1}{\phi_i^{r}-1}\right| > x-\frac{1}{2}.
\]%

 So we set $u_1 - \delta_H v_1=\nu^{r/4}, \gamma_1=\nu^r, u_2- \delta_H v_2=\nu^{rx/4}, \gamma_2=\nu^{rx}$. Then
for $i=1,2$ equation (\ref{eqdeg2:2i}) is satisfied. We also satisfy (\ref{eqdeg2:3i}) for these
values of $i$. Next we define $a_1$ and $b_1$ so that (\ref{eqdeg2:4i}) is satisfied for $j=1$.
Next let $\sigma$ be an embedding of $M$ into $\R$ extending to a real embedding of $GM$
and therefore corresponding to a real embedding of $G$. Then by assumption on $H$, we have that
$\sigma$ extends to a non-real embedding $\hat{\sigma}$ on $HM$.  Thus, without loss of generality,
for some $i =1,\ldots,s_{H}$ we have that
\[%
\hat\sigma(a_1-\delta_H b_1) = \hat \sigma\left (\frac{\nu^{rx}-1}{\nu^r-1}\right )=
\frac{\phi_i^{rx}-1}{\phi_i^r-1},
\]%
and therefore
\[%
\sigma(a_1^2-d_Hb_1^2)=\left | \frac{\phi_i^{rx}-1}{\phi_i^r-1}\right |^2 > x - \frac{1}{2},
\]%
leading to
\[%
1+ \sigma(a_1^2-d_Hb_1^2)^2 > x = \sigma(x) >1.
\]%
Thus we can satisfy (\ref{eqdeg2:7i}).

Let $\nu_3 \in B_K$, assume $\nu_3$ is not a root of unity, and let $\nu_3=u_3-\delta_Hv_3$ with $u_3, v_3 \in
O_{M}[\delta_H]$ be such that
\begin{equation}%
\label{eq:satisfy}%
\nu_3 - 1 \equiv 0 \mod 6a_Gx(1+(a_1^2-d_Hb_1^2)^2)
\end{equation}%
This can be done by Corollary \ref{cor:4power} and by Section 2.1.1 of \cite{Sh16}. Set $\gamma_3=
\nu_3^4$. Then (\ref{eqdeg2:2i}), (\ref{eqdeg2:3i}), and (\ref{eqdeg2:4.2i}) for $i=3$ are
satisfied. Finally, set $\gamma_4=\gamma_3^{x}, \nu_4=\nu_3^x$. In this case we can satisfy
(\ref{eqdeg2:2i}), (\ref{eqdeg2:3i}), and (\ref{eqdeg2:4.2i}) for $i=4$. We now observe that
\[%
a_2-\delta b_2=\frac{\gamma_4-1}{\gamma_3-1} =x + (\gamma_3-1)(u+\delta_H v)=x +(c_3-1-\delta_H
d_3)(u+v\delta_H),
\]%
where $u, v \in O_{K}$.  Thus (\ref{eqdeg2:9i}) will also be satisfied. The only remaining issue is  to note that
from the discussion of Section 2 we can rewrite all the equations and conditions as polynomial equations with
variables taking values in $O_{GK_{\infty}}$.  This can be done using Propositions 2.6 -- 2.8 of \cite{Sh36}.
\end{proof}%

We can now state the main theorem of this section.
\begin{theorem}
\label{thm:mainintdeg2}%
Let $K_{\infty}$ be a totally real possibly infinite algebraic extension of $\Q$. Let $U_{\infty}$ be a finite
extension of $K_{\infty}$ such that there exists an elliptic curve ${\tt E}$ defined over $U_{\infty}$ with
${\tt E}(U_{\infty})$ finitely generated and of a positive rank. Let $GK_{\infty}$ be an extension
of degree 2 over $K_{\infty}$. If $GK_{\infty}$ has no real embeddings, assume additionally that
$K_{\infty}$ has a totally real extension of degree 2.  Then $\Z$ is existentially definable and
HTP is unsolvable over the ring of integers of $GK_{\infty}$.
\end{theorem}%

\section{Diophantine Definition of Some Large Rational Subrings and Arbitrary Small Subrings  for
Extensions of Degree 2 of Totally Real Fields.}%
\setcounter{equation}{0}%
In this section we will prove a result analogous to Theorem \ref{thm:mainintdeg2} for some big and
arbitrary small subrings of $K_{\infty}$ under some assumptions on the field under consideration.
As in the case of the ring of integers most of the work has already been done for this case in
\cite{Sh36} but we will have to adjust notation and some details.
\begin{notationassumption} %
\label{not:degree2bigrings}%
We now bring back all the assumptions we had concerning $K_{\infty}$ as in Notation and Assumptions
\ref{not:biggerrings} and \ref{not:bigringtoreal}. We also continue to use notation and assumptions from
\ref{not:ext2integers} except for the first item which deals with $K_{\infty}$. Finally we add the following to
our notation and assumptions  list.
\begin{itemize}%
\item Let $\calZ_K$ be a subset of $\calV_K$ such that the primes of $\calZ_K$ do not split in the extension $G/K$
and do not divide the discriminant of $R(X)$.%

\item For any number field $M$ with $K \subset M \subset K_{\infty}$, assume that $[M:K]$ is odd. (By Lemma
\ref{le:nodegonefact} this assumption implies that primes of $\calZ_M$ remain inert in the extension $GM/M$.)%
\item Assume that at least two $K$-primes $\qq_1$ and $\qq_2$ lying above two different rational primes are not
ramified in $K_{\infty}$.%
\item Assume that the field $EHGK_{\infty}$ contains no roots of unity which are not already in
$EGK_{\infty}$. (The lemma below assures us that we can arrange this under our assumptions.)
\item Let $p$ be any rational prime without any factors in $\calZ_K$ and such that $p >
|\phi(\alpha_G)|$ for any $\phi$, an embeddings of  $K$ into its algebraic closure.
\end{itemize}%
\end{notationassumption}%
\begin{proposition}%
There exists a number field $H$ satisfying all the requirements from Notation and Assumptions
\ref{not:ext2integers} and  \ref{not:degree2bigrings}.
\end{proposition}
\begin{proof}%
Let $a_1, a_2 \in O_K$ be such that $\ord_{\qq_i}a_i=1$.  By the Strong approximation theorem we can find $d_H \in
O_K$ in such that its conjugates over $\Q$ have the right sign and such that $\ord_{\qq_i}(d_H-a_i) > 2$ for
$i=1,2$. The last requirement will make sure that $\ord_{\qq_i}d_H=1$ and  factors of two primes are ramified in
the extension $EHGK_{\infty}/EGK_{\infty}=EGK_{\infty}(\delta_H)/HGK_{\infty}$. This is enough to make sure that
this extension of degree 2 is not generated by any root of unity. (See Lemma 2.4 of \cite{Sh17}.)
\end{proof}%
To prove our results we will need the following technical propositions from \cite{Sh36}.
\begin{proposition}%
\label{prop:norm}%
Consider the following system of equations where $M$, as usual, is a number field with $K \subseteq
M \subset K_{\infty}$.
\begin{equation}%
\label{sys:norm}
\left \{%
\begin{array}{c}
{\mathbf N}_{EHGM/EGM}(\varepsilon)=1\\%
{\mathbf N}_{EHGM/HGM}(\varepsilon)=1%
\end{array} \right .%
\end{equation}%
We claim that  any for  $\varepsilon \in O_{EHGM,\calZ_{EHGM}}$ satisfying
(\ref{sys:norm}) we have that $\varepsilon^2 \in O_{HME}$.  Further, there is always $\varepsilon
\in U_{HME}$ such that it is not a root of unity and is a solution to the system.
\end{proposition}%
(See Lemma 5.3 and  Corollary 5.4 of \cite{Sh36} for a proof.)
\begin{lemma}%
\label{le:order}%
Let $x \in O_{GM,\calZ_{GM}}$, $x =y_0+y_1\alpha_G \equiv z \mod \mte{Z}$ in $O_{HEGM,\calZ_{HEGM}}$, where $z
\in O_{HEM,\calZ_{HEM}}$, $y_0,y_1 \in M$, $\mte{Z}$ is an integral divisor of $HEM$ without any factors in
$\overline \calZ_{HEM}$. Assume additionally that for any $\mte{t} \in \overline{\calZ}_{GEM}$ we have that
$\ord_{\mte{t}}x \leq 0$. Let $\displaystyle {\mathbf N}_{HEGM/\Q}(x)=\frac{X}{Y}$, where $X, Y \in
\Z$ and $(X,Y)=1$ in $\Z$. Let $Z= {\mathbf N}_{HEGM/\Q}(\mte{Z})$. Then $\displaystyle\frac{Y}{Z}{\mathbf
N}_{HEGM/\Q}(2\alpha y_1)$ is an integer.
\end{lemma}%
  (See Lemma 6.6 of \cite{Sh36} for proof.)\\

We are now ready to state the main result of this section whose proof is similar to the proof of
 Proposition 10.5 in \cite{Sh36}.  However, as usual, some details need to be changed.
\begin{theorem}%
\label{thm:2}
 There exists a polynomial equation $P(x,\bar{t}) \in O_K[x,\bar{t}]$ such that the following statements are true.%
\be%
\item For any $x \in O_{GK_{\infty},\calZ_{GK_{\infty}}}$, if $P(x,\bar{t}) =0$ for some
$\bar{t}=(t_1,\ldots,t_m)$ with $t_i \in O_{GK_{\infty},\calZ_{GK_{\infty}}}$, then $x \in O_{K_{\infty},\calZ_{K_{\infty}}}$.%
\item If $x \in \Z$ then there exists $\bar{t}=(t_1,\ldots,t_m)$ with $t_i \in O_{K, \calZ_K}$ such that $P(x,\bar{t}) =0$%
\ee%
\end{theorem}%
\begin{proof}%
Let $x_0, x_1 \in O_{G_{\infty},\calZ_{G_{\infty}}}, a_1, a_2, b_1, b_2, c_1, d_1, \ldots, c_4,
d_4, u, v \in O_{EG_{\infty},\calZ_{EG_{\infty}}}$, $\varepsilon_i, \gamma_i \in
O_{EHGK_{\infty},\calZ_{EHGK_{\infty}}}, i=1,\ldots,4,$ and assume the following conditions and
equations are satisfied.

\begin{equation}%
\label{eqdeg2:1}
x_1 = R(x_0),
\end{equation}

\begin{equation}%
\label{eqdeg2:2}
\left \{%
\begin{array}{c}
{\mathbf N}_{HEGK_{\infty}/EGK_{\infty}}(\varepsilon_i)=1, i=1,\ldots,4,\\%
{\mathbf N}_{HEGK_{\infty}/HGK_{\infty}}(\varepsilon_i)=1, i=1,\ldots,4,%
\end{array} \right .%
\end{equation}%

\begin{equation}%
\label{eqdeg2:3}%
\gamma_i = \varepsilon_i^2, i=1, \ldots,4,
\end{equation}%
\begin{equation}%
\label{eqdeg2:4}%
\frac{\gamma_{2j}-1}{\gamma_{2j-1}-1} =a_j-\delta_H b_j, j=1,2
\end{equation}%
\begin{equation}
\label{eqdeg2:4.2}%
\gamma_i=c_i+\delta_H d_i, i=1,\ldots, 4,
\end{equation}%
\begin{equation}%
\label{eqdeg2:7}%
1 \leq |\sigma(x_1)| \leq R(A(1)+\sigma(a_1^2-d_Hb_1^2)^2),
\end{equation}%
where $\sigma$ ranges over all  embeddings of $EGK_{\infty}$ into $\R$,
\begin{equation}%
\label{eqdeg2:9}%
x_1-(a_2-\delta_H b_2) =(c_3-1+\delta_H d_3)(u+v\delta_H),
\end{equation}%
\begin{equation}%
\label{eqdeg2:10}%
p^2x_1R(A(1)+(a_1^2-d_Hb_1^2)^2) \big{|} (c_3-1+\delta_H d_3).
\end{equation}%
Then, we claim,  $x_1 \in K_{\infty}$.

Conversely, we claim that if $x_0 \in \Z_{>1}$, the conditions and equations above can be satisfied
with $a_1, a_2$, $b_1, b_2, c_1, d_1, \ldots, c_4, \ldots, d_4, u, v \in O_{EK,\calZ_{EK}}$,
$\gamma_i, \varepsilon_i \in O_{EMHK,\calZ_{EH}}, i=1,\ldots,4$.

To prove the first claim, observe the following. Let $M$ be such that $GHEM$ contains $\alpha_G, \delta_H,\mu_E,
x_0,a_1,a_2$, $b_1,b_2,c_1, d_1, \dots, c_4, d_4, u, v, \varepsilon_1, \ldots, \varepsilon_4$. Then
given our assumptions on the fields under consideration, in the equations above we can replace
$EHGK_{\infty}$ by $EHGM$, $EGK_{\infty}$ by $EGM$, and finally $K_{\infty}$ by $M$, while the
equalities and other conditions will continue to be true, assuming we modify the prime sets by
choosing the primes in the finite extensions so that $O_{GK_{\infty},\calZ_{GK_{\infty}}}$ is the
integral closure of $O_{GM,\calZ_{GM}}$, $O_{EGK_{\infty},\calZ_{EGK_{\infty}}}$ is the integral
closure of $O_{EGM,\calZ_{EGM}}$, and $O_{EHGK_{\infty},\calZ_{EHGK_{\infty}}}$ is the integral
closure of $O_{EGHM,\calZ_{EHGM}}$ in $GK_{\infty}$, $EGK_{\infty}$ and $EHGK_{\infty}$
respectively. Then by Proposition \ref{prop:norm} we know that $\gamma_i \in HEM \subset
HEK_{\infty}$.  Since $\delta_H$ generates $HEM$ over $EM$ as well as $HEGM$ over $EGM$, we
conclude that $c_i,  d_i, \in O_{EM,\calZ_{EM}}$ for $i=1,\ldots, 4$. A similar argument tells us
that $a_1,b_1, a_2, b_2 \in O_{EM,\calZ_{EM}}$.

Next from (\ref{eqdeg2:1}) and Lemma \ref{le:conjugates} we conclude that for all $\mte{p} \in
\overline{\calZ}_{EGM}$ we have that $\ord_{\pp}x_1\leq 0$ and
\[%
\ord_{\pp}R(A(1)+ (a_1^2-d_Hb_1^2)^2)\leq 0.
\]%
From the definition of $A(1)$  and the fact that $a_1, b_1 \in EM$ - a totally real field, we have
that
\begin{equation}%
\label{eqdeg2:8}%
1 \leq  R(A(1) + \tau(a_1^2-d_Hb_1^2)^2),
\end{equation}%
where $\tau$ ranges over all non-real embeddings of $EGM$ into $\C$. Combining the bound equations
(\ref{eqdeg2:7}) and (\ref{eqdeg2:8}), and writing $x_1 = y_0 + y_1\alpha_G$, where $y_0, y_1\in O_{M,\calU_{M}}$,
we conclude by Lemma \ref{le:bound2}
\begin{equation}%
\label{eqdeg2:11}%
|{\mathbf N}_{EGM/\Q}(y_1)| \leq      |{\mathbf N}_{EGM/\Q}(x_1){\mathbf N}_{EGM/\Q}(R(A(1)+
(a_1^2-d_Hb_1^2)^2))|.
\end{equation}%

Let $\mte{D}=\mte{n}_{EGHM, \overline \calZ_{EGHM}}(c_3-1+\delta_H d_3)$.  Note that since
$c_3-1+ \delta_H d_3 \in HEM$ and  $\overline \calZ_{HEGM}$ is closed under conjugation over $\Q$ and thus
over $HEM$, we can regard $\mte{D}$ as a divisor of $HEM$.   Observe further that from (\ref{eqdeg2:9}), we
have that $\mte{D} \mbox{ divides }\mte{n}_{GHEM}(x_1-(a_2-b_2\delta_H))$.  Let
\[%
D= |{\mathbf N}_{HEGM/\Q}(\mte{D})| \in \Z_{>0},
\]%
let
\[%
|{\mathbf N}_{HEGM/\Q}(x_1)|=\frac{X}{Y},
\]%
and let
\[%
|{\mathbf N}_{HEGM/\Q}(A(1)+R(a_1^2-d_Hb_1^2)^2)|=\frac{U}{V},
\]%
where $X,Y,U,V \in \Z_{>0}, (X,Y)=1, (U,V)=1$, and $X, U$ are not divisible by any rational primes with factors in
$\calZ_{HEGM}$. Then from (\ref{eqdeg2:10}) we have that
\begin{equation}%
\label{eq:XU}
\norm_{EGHM/\Q}(2\alpha_G)XU <  \norm_{EGHM/\Q}(p^2)XU < D.
\end{equation}%
Note that by Lemma \ref{le:nodegonefact}  we have that all the primes of $\calZ_M$ do not split in the extension
$GM/M$ and therefore by Lemma \ref{le:order}, on the one hand we have that
\[%
\frac{Y{\mathbf N}_{GHME/\Q}(2\alpha_G y_1)}{D} \in \Z,
\]%
and therefore
\[%
|Y{\mathbf N}_{GHME/\Q}(2\alpha y_1)| \geq D\mbox{ or } y_1=0.
\]%
On the other hand, combining (\ref{eqdeg2:11}) and (\ref{eq:XU}), we have that
\[%
|Y{\mathbf N}_{GHME/\Q}(2\alpha_G y_1)| \leq |\norm_{EGHM/\Q}(2\alpha_G)XU| < D.
\]%
Thus $y_1$ is 0 and $x_1 \in M$.\\

We will now show that assuming that $x_0 >1$ is a natural number, we can satisfy all the equations
and conditions (\ref{eqdeg2:1})--(\ref{eqdeg2:10})with all the variables ranging over the
appropriate sets. Observe that by (\ref{eqdeg2:1}), we have that $x_1$ is also a natural number.
Let $\nu \in U_{HE} \cap O_K[\delta_H, \mu_E]$ be a solution to (\ref{eqdeg2:2}) such that it is
not a root of unity. Such a solution exists by Proposition \ref{prop:norm} and by Section 2.1.1 of
\cite{Sh16}. Let $\{\phi_1,\ldots,\phi_{s_{HE}}\}$ be a set containing a representative from every
complex-conjugate pair of non-real conjugates of $\nu$. By Lemma \ref{le:closeto1}, we can find a
positive integer $r \equiv 0 \mod 2$ such that for all $i=1,\ldots, s_{HE}$ we have that
\[%
\left |\frac{\phi_i^{rA}-1}{\phi_i^r-1}-A\right| < \frac{1}{2},
\]%
where $A=A(x_1) +1$, and thus,
\[%
\left |\frac{\phi_i^{rA}-1}{\phi_i^{r}-1}\right| > A-\frac{1}{2}>A(x_1).
\]%

 So we set $\varepsilon_1=\nu^{r/2}, \gamma_1=\varepsilon^r, \varepsilon_2=\varepsilon^{rA/2},
\gamma_2=\varepsilon^{rA}$. Then for $i=1,2$ the system (\ref{eqdeg2:2}) is satisfied. We also satisfy
(\ref{eqdeg2:3}) for these values of $i$. Next we define $a_1$ and $b_1$ so that (\ref{eqdeg2:4}) is satisfied for
$j=1$. Next let $\sigma$ be an embedding of $K$ into $\R$ extending to a real embedding of $G$ and
therefore to a real embedding of $GE$. Then by assumption on $H$, we have that $\sigma$ extends to a non-real
embedding $\hat{\sigma}$ on $HE$.  Thus, without loss of generality, for some $i =1,\ldots,s_{HE}$ we have
that
\[%
\hat\sigma(a_1-\delta_H b_1) = \hat \sigma\left (\frac{\varepsilon^{rA}-1}{\varepsilon^r-1}\right )=
\frac{\phi_i^{rA}-1}{\phi_i^r-1},
\]%
and therefore
\[%
\sigma(a_1^2-d_Hb_1^2)=\left | \frac{\phi_i^{rA}-1}{\phi_i^r-1}\right |^2 > A(x_1)^2 >A(x_1),
\]%
leading to
\[%
Q(A(1)+ \sigma(a_1^2-db_1^2)^2) > x_1 = \sigma(x_1) >1.
\]%
Thus we can satisfy (\ref{eqdeg2:7}).

Now let $\varepsilon_3$ to be a solution to (\ref{eqdeg2:2}) in $O_{K}[\delta_H, \mu_E]$ such that $\gamma_3=
\varepsilon_3^2 \in U_{HE} \cap O_{K}[\delta_H, \mu_E]$, (\ref{eqdeg2:3}), (\ref{eqdeg2:4.2}) for
$i=3$, and (\ref{eqdeg2:10}) are satisfied. Again this can be done by Proposition \ref{prop:norm}
and by Section 2.1.1 of \cite{Sh16}. Finally, set $\varepsilon_4=\varepsilon_3^{x_1},
\gamma_4=\gamma_3^{x_1}$. In this case we can satisfy (\ref{eqdeg2:2}), (\ref{eqdeg2:3}) for
$i=4$.

We now  observe that
\[%
a_2-\delta_H b_2=\frac{\gamma_4-1}{\gamma_3-1} =x_1 + (\gamma_3-1)(u+\delta_H v)=x_1 +(c-1-\delta_H
d)(u+v\delta_H),
\]%
where $u, v \in O_{GM,\calU_{GM}}[\mu_E]$. Thus (\ref{eqdeg2:9}) will also be satisfied. As a last
step we select $c_1, d_1, \ldots,c_4, d_4 \in O_{M}$ so that (\ref{eqdeg2:4.2}) is satisfied. The
only remaining issue is to note that we can rewrite all the equations and conditions as polynomial
equations with variables taking values in $GK_{\infty}$ (see Propositions 2.6 -- 2.8 of \cite{Sh36}
again), and also to  observe that we can require $x_0 + 1,\ldots,x_0 + n_E$ to satisfy the equations
above. Making sure that $x_0, x_0 + 1, \ldots, x_0+n_E \in K_{\infty}$ is enough by Lemma 5.1 of
\cite{Sh1} to insure that $x_0$ is in $K_{\infty}$.
\end{proof}%

Our next task is to go down to $\Q$.  Unfortunately as in Section \ref{sec:real} there are technical
complications which will force us  to modify the prime sets allowed in the denominators of the
divisors of ring elements.    The problem lies in the fact that we might have to add a finite set of
 primes to  $\calZ_K$ in order to go down to $\Q$.  In the case of the totally real fields we just
needed to make sure that in Proposition \ref{prop:thenbelow} the ``key'' variable $x_j$ did not have
the ``extra'' primes in the denominator of its divisor.  This was accomplished by using the fact
that any finite set of primes   over the field under consideration was boundable.  Unfortunately,
in this case we have another problem to worry about: extra solutions to the norm equations.  This
is the same problem which we encountered in \cite{Sh36}.  To avoid the extra solutions we will have
to introduce another extension $\bar EK$ of $K$  which will be totally real and linearly disjoint
from $GHEK_{\infty}$ over $K$.
\begin{notationassumption}%
\label{not:more}
Below we list our additional notation and assumptions.
\begin{itemize}%
\item Let $\bar E$ be a totally real number field of prime degree $n_{\bar E} >  [K^{\Gal}:\Q]$ with $(n_{\bar E}, n_E)=1$.%
\item Let $\calN_K$ be a set of $K$-primes remaining inert in the extension $\bar EK/K$.%
\item Let $\calT_K$ be a set of $K$ primes such that $O_{K, \calN_K \cup \calT_K} \cap \Q$ has a
Diophantine definition over     $O_{K, \calN_K \cup \calT_K}$.  (Such a set $\calT_K$ exists by
Corollary 7.6.1 of \cite{Sh34}.)%
\item Let $\calS_K$ be a finite set of $K$ primes and note that by Proposition \ref{prop:finmany} we
also have that       $O_{K, \calN_K \cup \calT_K \cup \calS_K} \cap \Q$ has a Diophantine
definition over
    $O_{K, \calN_K \cup \calT_K \cup \calS_K}$%
\item Assume that $\calT_K \cup \calN_K \cup \calS_K = \calZ_K$.
\end{itemize}%
\end{notationassumption}%

Given our assumptions we immediately have the following corollary.

\begin{corollary}%
\label{cor:downreal2}%
 $O_{GK_{\infty},  \calZ_{GK_{\infty}}} \cap \Q$ has a Diophantine definition over $O_{GK_{\infty}, \calZ_{GK_{\infty}}}$.%
\end{corollary}

From this corollary standard techniques produce the following consequences.

\begin{corollary}%
\label{cor:downreal2.1}%
\be
\item For any archimedean or non-archimedean topology of $GK_{\infty}$ we can select $\calS_K$ so
that  $O_{GK_{\infty}, \calZ_{K_{\infty}}}$ has an infinite Diophantine subset discrete
in this topology.
\item $\Z$ is definable over $O_{GK_{\infty},\hat \calZ_{K_{\infty}} \cup \calT_K \cup \calS_K}$ and
HTP is not decidable over $O_{GK_{\infty},\hat \calZ_{K_{\infty}}\cup \calT_K\cup \calS_K }$.%
\ee
\end{corollary}%

(See, \cite{PS}, \cite{Sh36} or \cite{Sh34}.)\\

Finally we restate our results in the following form.
\begin{theorem}%
\label{thm:bigrings2}
Let $K_{\infty}$ be a totally real possibly infinite algebraic extension of $\Q$, normal over some
number field $K$, with an odd rational prime $p$ of 0 degree index relative to $K_{\infty}$
such that $p > [K^{\Gal}:\Q]$, and with a 0 degree index for 2. Let $U_{\infty}$ be a finite
extension of $K_{\infty}$ such that there exists an elliptic curve ${\tt E}$ defined over
$U_{\infty}$ with ${\tt E}(U_{\infty})$ finitely generated and of a positive rank. Let $GK_{\infty}$
be an extension of degree 2 over $K_{\infty}$.
\be%
\item $GK_{\infty}$ contains a big ring where $\Z$ is existentially definable and HTP is unsolvable. %
\item $\Z$ is definable and HTP is unsolvable over any small subring of $GK_{\infty}$.%
\ee
\end{theorem}%
\begin{proof}%
\be%
\item Let $K \subset K_{\infty}$ and $p$ be as in the statement of the theorem. Next pick a totally
real cyclic extension $\bar E$ of $\Q$ so that $[\bar E:\Q]$ is a prime number greater than
$[K^{\Gal}:\Q]$ and is different from $p$. Let $\calM_K$ be a set of primes inert in the extension
$\bar EK/K$. Let $\calT_K$ be defined as in Notation and Assumptions \ref{not:more} so that $O_{K,
\calM_K \cup \calT_K} \cap \Q$ has a Diophantine definition over $O_{K, \calM_K \cup \calT_K}$, and
let $\calS_K$ be an arbitrary set of primes. Then by Lemma A9 of \cite{Sh36} there exists a totally
real cyclic extension $E$ of $\Q$ of degree $p$ such that no prime of $\calT_K \cup \calS_K$ splits
in the extension $EK/K$. Further $E$ and $\bar E$ will satisfy Notation and Assumptions
\ref{not:biggerrings} and \ref{not:more}. Let $\calN_K \subset \calM_K$ be the set of
$\calM_K$-primes inert in the extension $EGK/K$. Now let $\calZ_K=\calN_K \cup \calS_K \cup
\calT_K$. This set can be infinite since it can contain all the primes inert in the cyclic
extension $E\bar EGK/K$. It might also happen that some primes of $\calT_K$ or $\calS_K$ split in
$GK/K$ or divide the discriminant of $R(X)$.  In this case we will have to use the fact that we can
bound any finite set of primes in $GK_{\infty}$ as in the totally real case.  Now observe that minus
this complication,  $\calZ_K$ satisfies Notation and Assumptions \ref{not:more} and thus the first
assertion of the lemma follows.

\item This assertion of the lemma can be handled in almost the same way as the first assertion.
Indeed suppose we start in $O_{GK_{\infty}, \calS_{GK_{\infty}}}$ instead of $O_{GK_{\infty},
\calZ_{GK_{\infty}}}$. Then the only difference will be that when we ``descend" to $K$ we will find
ourselves in the ring $O_{K,\calS_K}$.  Thus, we can use Proposition \ref{prop:finmany}  to take
the ``integer'' route down to $\Z$ without worrying about primes in $\calT_K$.
\ee%
\end{proof}%

 If we assume additionally that the set of rational primes with 0 degree index with respect to
$K_{\infty}$ is infinite, then we have a simplified version of the statements above.
\begin{theorem}%
\label{thm:bigrings3}
 Let $K_{\infty}$ be a totally real possibly infinite algebraic extension of $\Q$, normal over some
number field. Let $U_{\infty}$ be a finite extension of $K_{\infty}$ such that there exists an
elliptic curve ${\tt E}$ defined over $U_{\infty}$ with ${\tt E}(U_{\infty})$ finitely generated.
Let $GK_{\infty}$ be an extension of degree 2 over $K_{\infty}$. Assume that the set of rational
primes with 0 degree index with respect to $K_{\infty}$ is infinite and includes 2.  Then for
every $\varepsilon >0$ there exists a number field $K \subset K_{\infty}$ and a set $\hat \calZ_K
\subset \calP(K)$ of density (natural or Dirichlet) bigger than $1/2 -\varepsilon$ such that $\Z$
is existentially definable and HTP is unsolvable in the integral closure of $O_{K,\hat \calZ_K}$ in
$GK_{\infty}$.
\end{theorem}%
\begin{proof}%
 The proof of this part of the theorem is completely analogous to the proof of Theorem 11.8 of
\cite{Sh36}.

\end{proof}%

\section{Using Rank One Curves in Infinite Extensions.}
In this section we will see to what extent we can duplicate the results in \cite{Po2} and \cite{PS}
over an infinite algebraic extension of $\Q$.  Some of our  assumptions and notation in this
section will be different from the ones we used above.

\begin{notationassumption}%
The following assumptions and notation will be different or new in this section.%
\begin{itemize}%
\item Let $K$ be a number field. (Note that we no longer assume that $K$ is totally real.)
\item Let $K_{\infty}$ be an algebraic extension of $K$.%
\item Let ${\tt E}$ be an elliptic curve defined over $K$ such that $\rank({\tt E}(K))=1$ and
${\tt E}(K_{\infty})={\tt E}(K)$. (Note the new rank assumption.)
\item Let $Q$ be a generator of ${\tt E}(K)$ modulo the torsion group.
\item For a rational prime $t$ let $\F_t$ be a finite field of $t$ elements.
\item Let $p \not = q \in \calP(\Q) \setminus \{2\}$. Let $\pp, \qq$ be $K$-primes above $p$ and $q$ respectively
with $f(\pp/p)=f(\qq/q)=1$. Assume additionally that $y_1(P) \not \equiv 0 \mod \pp$ and $y_1(P) \not \equiv 0 \mod
\qq$ while integrality at $\pp$ and $\qq$ is definable in $K_{\infty}$. %
\item Let $M=\#{\tt E}(\F_p)\#{\tt E}(\F_q)pq$.%
\end{itemize}%
\end{notationassumption}%
As we will see below we will be able to transfer almost seamlessly facts from the finite case to the
infinite case. First we review some technical details of the original results in \cite{Po2} and
\cite{PS}.

\begin{theorem}%
There exists a sequence of rational primes $\{\ell_i\}$, and sets $\calA_K, \calB_K \subset \calP(K)$ satisfying
the following properties.
\be%
\item The natural (and Dirichlet) density of $\calA_K$ and $\calB_K$ is 0.%
\item $\calA_K \cap \calB_K = \emptyset$.%
\item For any $\calW_K \subset \calP(K)$ such that $\calA_K \subseteq \calW_K$ and such that $\calW_K \cap
\calB_K=\emptyset$ we have that ${\tt E}(O_{K,\calW_K})=\{\pm\ell_iP:  i\in \Z_{>0}\} \cup \{\mbox{ finite
set }\}$.%
\item The set $\{x_{\ell_i}(Q): i \in \Z_{>0}\}$ is discrete in every $\pp$-adic and archimedean topology of $K$.
\ee%
\end{theorem}%
This theorem follows Lemma 3.10 and Proposition 3.15 of \cite{PS}.  Now, given our assumptions on the behavior of
${\tt E}$ over $K_{\infty}$ and the fact that any non-archimedean or archimedean topology of $K_{\infty}$ must be an
extension of the corresponding topology on $K$, we immediately obtain the following corollary.
\begin{corollary}%
There exists a sequence of rational primes $\{\ell_i\}$, and sets $\calA_K, \calB_K \subset \calP(K)$ satisfying
the following properties.
\be%
\item The natural (and Dirichlet) density of $\calA_K$ and $\calB_K$ is 0.%
\item $\calA_K \cap \calB_K = \emptyset$.%
\item For any $\calW_K \subset \calP(K)$ such that $\calA_K \subseteq \calW_K$ and such that $\calW_K \cap
\calB_K=\emptyset$ we have that ${\tt E}(O_{K_{\infty},\calW_{K_{\infty}}})=\{\pm\ell_iP:  i\in \Z_{>0}\} \cup \{\mbox{
finite set }\}$.%
\item The set $\{x_{\ell_i}(Q): i \in \Z_{>0}\}$ is discrete in every $\pp$-adic and archimedean topology of
$K_{\infty}$.
\ee%
\end{corollary}%
We now proceed to another result in \cite{PS} which constructs (indirectly) a Diophantine model of
$(\Z, +, \times)$ over $O_{K,\calW_K}$, where $\calW_K$ will satisfy somewhat different conditions
as will be described below. We start with a sequence of propositions from \cite{PS} (Lemma 3.16,
Corollary 3.17, Lemma 3.20):

\begin{lemma}%
Let $B=\{\,2^n+n^2 : n \in \Z_{\ge 1}\,\}$. Then multiplication admits a positive existential
definition in the structure $\calZ:=(\Z_{\ge 1},1,+,B)$. (Here $B$ is considered as an unary
predicate.)
 \end{lemma}%

\begin{corollary}%
\label{C:positive existential model}%
The structure $(\Z,0,1,+,\cdot)$ admits a positive existential model in the structure $\calZ$.
\end{corollary}   %
\begin{lemma}
\label{L:xdifference}%
Let $\ttt, t$ stand for $\pp, p$ or $\qq, q$ respectively.  Then if $m \in \Z_{\ge 1}$, we have that
\[%
\ord_\ttt(x_{mM+1}-x_1) = \ord_\ttt(x_{M+1}-x_1) + \ord_t m.
\]%
\end{lemma}%
\begin{proposition}%
\label{prop:overK}
There exists a computable sequence of rational primes $\{\ell_i\}$, and sets $\calA_K, \calB_K \subset \calP(K)$
satisfying the following properties.
\be%
\item The natural (and Dirichlet) density of $\calA_K$ and $\calB_K$ is 0.%
\item $\calA_K \cap \calB_K = \emptyset$.%
\item For any $\calW_K \subset \calP(K)$ such that $\calA_K \subseteq \calW_K$ and such that $\calW_K \cap
\calB_K=\emptyset$ we have that ${\tt E}(O_{K_{\infty},\calW_{K_{\infty}}})={\tt E}(O_{K,\calW_K})=\{\pm\ell_iP:  i\in \Z_{>0}\}
\cup \{\mbox{ finite set }\}$.%
 \item   \label{it:order} The highest power of $p$ dividing $(\ell_i-1)/M$ is $p^i$, and  $ i \in B$
   iff and only if $q$ divides $(\ell_i-1)/M$%
\ee%
\end{proposition}
\begin{proposition}
\label{prop:modelB}
Let $A:=\{x_{\ell_1},x_{\ell_2},\dots\}$. Then $A$ is a Diophantine model of $\calZ$ over
$\OO_{K_{\infty},\calW_{K_{\infty}}}$, via the bijection $\phi\colon \Z_{\ge 1} \to A$ taking $i$ to $x_{\ell_i}$.
\end{proposition}
\begin{proof}
The set $A$ is Diophantine over $\OO_{K_{\infty},\calW_{K_{\infty}}}$ by Proposition \ref{prop:overK}. Further we
have
\begin{align*}%
i \in B \quad&\iff\quad \text{$q$ divides $(\ell_i-1)/M$} &&\text{(by Proposition \ref{prop:overK} again)}\\
\quad&\iff\quad \ord_\qq(x_{\ell_i}-x_1) > \ord_\qq(x_{M+1}-x_1),%
\end{align*}%
by Lemma~\ref{L:xdifference}. The latter inequality is a Diophantine condition on
$x_{\ell_i}$ over $K_{\infty}$ by our assumption that integrality is definable at $\qq$ over $K_{\infty}$. Thus
the subset $\phi(B)$ of $A$ is Diophantine over $\OO_{K_{\infty},\calW_{K_{\infty}}}$.

Finally, for $i \in \Z_{\ge 1}$, Lemma \ref{L:xdifference} and assertion \ref{it:order} of Proposition
\ref{prop:overK}  imply $\ord_\pp(x_{\ell_i}-x_1) = c + i$, where the integer $c=\ord_\pp(x_{M+1}-x_1)$ is
independent of $i$. Therefore, for $i,j,k \in \Z_{\ge 1}$, we have
\[%
i+j=k \quad\iff\quad \ord_\pp(x_{\ell_i}-x_1) + \ord_\pp(x_{\ell_j}-x_1) = \ord_\pp(x_{\ell_k}-x_1) + c.
\]%
Since integrality at $\pp$ is also definable over $K_{\infty}$ by our assumptions, it follows that the graph of $+$
corresponds under $\phi$ to a subset of $A^3$ that is Diophantine over $\OO_{K_{\infty},\calW_{K_{\infty}}}$.  Thus
$A$ is a Diophantine model of $\calZ$ over $\OO_{K_{\infty},\calW_{K_{\infty}}}$.
\end{proof}%
We can now combine Proposition \ref{prop:modelB} and Corollary \ref{C:positive existential model} to obtain the
main result of this section.

\begin{theorem}
$\OO_{K_{\infty},\calW_{K_{\infty}}}$ has Diophantine model of $\Z$ and therefore HTP is undecidable over
$\OO_{K_{\infty},\calW_{K_{\infty}}}$.
\end{theorem}

We summarize the discussion above in the following theorem.
\begin{theorem}
\label{thm:rank1}
Let $K_{\infty}$ be an algebraic extension of $\Q$ such that there exists an elliptic curve ${\tt E}$ defined over
$K_{\infty}$ with ${\tt E}(K_{\infty})$ of rank 1 and finitely generated.  Fix a Weirstrass equation for ${\tt E}$ and a
number field $K$ containing all the coefficients of  the Weierstrass equation and the coordinates of all the
generators of ${\tt E}(K_{\infty})$.  Assume that $K$ has two odd relative degree one primes $\pp$ and $\qq$ such that
integrality is definable at $\pp$ and $\qq$ over $K_{\infty}$.
\be%
\item There exist a set $\calW_K$ of $K$-primes of natural density 1 such that over
$O_{K_{\infty},\calW_{K_{\infty}}}$  there exists an infinite Diophantine set simultaneously
discrete in all archimidean and non-archimedean topologies of $K_{\infty}$.

\item There exist a set $\calW_K$ of $K$-primes of natural density 1 such that over
$O_{K_{\infty},\calW_{K_{\infty}}}$ in $K_{\infty}$ there exists a Diophantine model of $\Z$ and therefore HTP is
not solvable over $O_{K_{\infty},\calW_{K_{\infty}}}$.
\ee%
\end{theorem}%

\section{Examples}
In this section we discuss some examples of elliptic curves and fields to which our results are applicable.  Our
primary source is \cite{M4}.   Using Notation and Assumptions \ref{not:integers}, let $K=\Q$ and let $F=\Q(\sqrt{-7})$.
Let ${\tt E}$ be the elliptic curve defined by the equation $y^2+y=x^3-x$.  A direct calculation shows that this
elliptic curve does not have complex multiplication.  Let $K_{\infty}$ be the unique cyclotomic $\Z_5$ extension of
$\Q$.  Then from the example in Section 1 of \cite{M4} we have that $\rank({\tt E}(FK_{\infty}))=1$, and by Proposition
\ref{prop:fingen} the Mordell-Weil group of ${\tt E}(FK_{\infty})$ is finitely generated.  Further $K_{\infty}$ has a totally
real extension of degree 2.  Thus by Theorem \ref{thm:mainint} and Theorem \ref{thm:mainintdeg2}, $\Z$ is existentially
definable and HTP is unsolvable in the ring of integers of $K_{\infty}$ and any extension of degree 2 of $K_{\infty}$.

We next consider the big ring situation. Observe that only one rational prime ramifies in $K_{\infty}$, $K_{\infty}$ is
Galois over $K$, and degrees of all the number fields contained in $K_{\infty}$ are powers of $5$. Thus, integrality is
definable over $K_{\infty}$ at all but finitely many primes by Corollary \ref{cor:deg}. Further, it then follows by
Theorems \ref{thm:mainnonint0}, \ref{thm:smallring}, \ref{thm:bigrings2}, and  \ref{thm:bigrings3}
that for all small and some big subrings $R$ of $K_{\infty}$ or its arbitrary extension of degree
2, $\Z$ is existentially definable and HTP is unsolvable over $R$.  We must point out here that
these conclusions concerning small and big rings (but not rings of integers) also follow from
\cite{Sh36} since $K_{\infty}$ is an abelian extension of $\Q$ with finitely many ramified rational
primes.  Note also that the results concerning the ring of integers of $K_{\infty}$ cannot be
obtained directly from the theorem of Denef concerning infinite extensions because it requires the
elliptic curve  of positive rank over $\Q$ keeping the same rank in an infinite totally real
extension.

Finally we can exploit the fact that the elliptic curve is of rank one in $FK_{\infty}$ to conclude that $FK_{\infty}$
has a ``very large'' ring $R$ (i.e. a ring which is an integral closure of a big subring of a number field with the
natural density of inverted primes equal to 1) which has a Diophantine model of $\Z$ and unsolvable HTP.

\section{Appendix}
\begin{lemma}%
\label{le:denom} Let $M$ be a number field.  Let $z_1, z_2 \in M$ be such that $\mathfrak d_M(z_1)$ and $\mathfrak
d_M(z_2)$ have no common factors.  Then we can write $\displaystyle z_1=\frac{a_1}{b_1}, z_2=\frac{a_2}{b_2}$,
where $a_1, a_2, b_1, b_2 \in O_M$, $(b_1,b_2) =1$ in $O_M$ and $(\mte{n}_M(a_i),\mte{d}_M(z_i))=1$.
\end{lemma}%
\begin{proof}%
By the Strong Approximation Theorem there exists $b_1 \in O_M$ such that for any prime $\pp \in
\calP(M)$ occurring in $\mathfrak d_M(z_1)$  we have that $\ord_{\pp}b_1 = \ord_{\pp}\mathfrak
d_M(z_1)$ and for any prime $\pp \in \calP(M)$ occurring in    $\dd_M(z_2)$, we have that
$\ord_{\pp}b_1 =0$.     Further, also by the Strong Approximation Theorem, there exists $b_2 \in
O_M$ such that for any prime $\pp \in \calP(M)$ occurring in $\dd_M(z_2)$ we have that
$\ord_{\pp}b_2 = \ord_{\pp}\dd_M(z_2)$ and for any prime $\pp \in \calP(M)$ such that
$\ord_{\pp}b_1 >0$ we have that $\ord_{\pp}b_2 = 0$.  Now we have that $b_iz_i \in O_M$, $(b_1,b_2)
=1$ and $(\nn_M(a_i),\dd_M(z_i))=1$.
\end{proof}%
\begin{lemma}%
\label{le:nodegonefact} %
Let $K$ be a number field. Let $E, G$ be two non-trivial Galois extensions of $K$ such that $([E:K],[G:K])=1$. Let
$\pp_K$ be a prime of $K$ such that $\pp_K$ does not have a degree one factor in $E$. Let $\pp_G$ be the $G$-prime
above $\pp_K$. Then $\pp_G$ does not have a relative degree one factor in $GE$.
\end{lemma}%
\begin{proof}%
Given our assumption on the degrees of the extensions involved, we have that $E \cap G=K$.  Consequently, since
both extensions are Galois, we conclude that $E$ and $G$ are linearly disjoint over $K$ and therefore
$[EG:G]=[E:K]$ and $[EG:G]=[G:K]$.   Further by Lemma B.3.7 of \cite{Sh34} we have that $EG/K$ is Galois and hence
every factor $\pp_G$ of $\pp_K$ in $G$ has the same number of factors, relative  and ramification degrees in $EG$
with all the three numbers dividing $[EG:G]=[G:K]$.  Similarly, every factor $\pp_E$ of $\pp_K$ in $E$ has the
same number of factors, relative and  ramification degrees in $EG$ with all the numbers dividing $[EG:E]=[E:K]$.
Let $\pp_{EG}$ be an $EG$-factor of $\pp_K$.  Let $\pp_E$ and  $\pp_K$ lie below $\pp_{EG}$ in
$E$ and $K$ respectively.  Then $f(\pp_{EG}/\pp_E)f(\pp_E/\pp_K) = f(\pp_{EG}/\pp_G)f(\pp_G/\pp_K)$.
 Further, $f(\pp_{EG}/\pp_E), f(\pp_G/\pp_K)$ are divisors of $[G:K]$ and thus are pairwise
relatively prime to $f(\pp_{EG}/\pp_G), f(\pp_E/\pp_K)$ which are divisors of $[E:K]$.  Therefore
by the Fundamental Theorem of Arithmetic we have $f(\pp_E/\pp_K) = f(\pp_{EG}/\pp_G)$.  By
assumption we have that $f(\pp_E/\pp_K) >1$, and therefore we also have that $f(\pp_{EG}/\pp_G)
>1$.
\end{proof}%
The following lemma is an expanded version of an argument from \cite{schmidwa} and \cite{shasp}.
\begin{lemma}%
\label{le:rootofunity}%
Let $G$ be a number field.  Let $x, y \in O_G$ be such that for some $d \in O_G$ we have that $x^2-dy^2=1$.  Let
$\delta$ be an element of the algebraic closure of $\Q$ be such that $\delta^2=d$ and assume that $\xi=x-\delta y$
is a root of unity.    Then $\xi^4=1$.
\end{lemma}%
\begin{proof}%
Observe that $\xi, \xi^{-1}$ and $x=\frac{\xi+\xi^{-1}}{2}$ are
algebraic integers. This is however impossible unless $\xi$ is rational or  $\xi +\xi^{-1}=0$, by
Proposition 2.16 of \cite{Wash}. Thus, $\xi^4=1$.
\end{proof}%
\bibliography{mybib}%

\end{document}